\documentclass{amsart}

\usepackage{epsfig,graphicx,amsmath,amsthm,amsfonts}
\usepackage{xcolor}
\usepackage{transparent}
\usepackage{amssymb}
\usepackage{mathrsfs}
\usepackage{amsxtra}
\usepackage{dsfont}

\newcommand{\brr}{\colon\!}

\newcommand{\charf}[1]{\chi_{_{#1}}}

\newcommand{\mbf}[1]{\mathbf{#1}}
\newcommand{\fat}[1]{\mathds{#1}}
\newcommand{\hsm}[1]{\mathscr{#1}}
\newcommand{\EE}{\fat{E}}
\newcommand{\NN}{\fat{N}}

\newcommand{\RR}{\fat{R}}
\newcommand{\RRplus}{\RR_{_{\geq 0}}}
\renewcommand{\SS}{\fat{S}}

\newcommand{\power}[1]{\mathbf{2}^{#1}}

\newcommand{\metr}[1]{\mathtt{Metr}\left(#1\right)}
\newcommand{\res}[1]{\left|_{#1}\right.}
\newcommand{\hau}[1]{#1_{_{hau}}}

\newcommand{\diag}[1]{\mathtt{diag}_{#1}}	
\newcommand{\length}[1]{\Lambda\left(#1\right)} 
\newcommand{\paths}[1]{\hsm{P}\left(#1\right)} 
\newcommand{\floor}[1]{\left\lfloor #1\right\rfloor}

\newcommand{\minP}{\mbf{0}}
\newcommand{\maxP}{\mbf{0}^\ast}
\newcommand{\zapprox}[2]{\left\lfloor #2\right\rfloor_{#1}}

\newcommand{\cub}[1]{\square^{#1}}

\newcommand{\flip}[2]{[#2]_{_{#1}}}

\newcommand{\THEN}{\;\Rightarrow\;}
\newcommand{\IFF}{\;\Leftrightarrow\;}
\newcommand{\minus}{\smallsetminus}
\newcommand{\symplus}{\vartriangle}

\newcommand{\ep}{\hfill $\square$}
\newcommand{\cl}[1]{\overline{#1}}

\newcommand{\inv}{^{{\scriptscriptstyle -1}}}
\newcommand{\set}[2]{\left\{#1\,\Big|\,#2\right\}}

\newcommand{\card}[1]{\left|#1\right|}

\newcommand{\norm}[1]{\left\Vert #1\right\Vert}

\newcommand{\diam}[1]{\mathtt{diam}\!\left(#1\right)}
\newcommand{\dist}[2]{\mathtt{dist}\!\left(#1,#2\right)}
\newcommand{\med}{\mathrm{med}}

\newcommand{\half}[1]{\mathfrak{h}_{#1}}
\newcommand{\wall}[1]{\mathfrak{w}_{#1}}
\newcommand{\cube}[2]{\mathtt{Cube}_{#1}(#2)}

\newcommand{\cubeone}[2]{\mathtt{Cube}^{#1}_1(#2)}
\newcommand{\cubeinfty}[2]{\mathtt{Cube}^{#1}_\infty(#2)}
\newcommand{\cubep}[2]{\mathtt{Cube}^{#1}_p(#2)}

\newcommand{\ellone}[1]{\mbf{\Delta}_1^{#1}}
\newcommand{\ellinfty}[1]{\mbf{\Delta}_\infty^{#1}}
\newcommand{\ellp}[1]{\mbf{\Delta}_p^{#1}}

\newtheorem{theorem}{Theorem}[subsection]
\newtheorem{proposition}[theorem]{Proposition}
\newtheorem{lemma}[theorem]{Lemma}
\newtheorem{corollary}[theorem]{Corollary}

\newtheorem{definition}[theorem]{Definition}

\theoremstyle{definition}
\newtheorem{example}[theorem]{Example}
\newtheorem{question}[theorem]{Question}

\theoremstyle{remark}
\newtheorem{remark}[theorem]{Remark}

\begin{document}

\title{Injective Metrizability and the Duality Theory of Cubings}

\author{Jared Culbertson}
\address{Sensors Directorate, Air Force Research Laboratory, 2241 Avionics Circle, Building 620
Wright--Patterson Air Force Base, Ohio 45433-7302, USA.}
\email{jared.culbertson@us.af.mil}

\author{Dan P. Guralnik}
\address{Electrical \& Systems Engineering Dept., University of Pennsylvania, 200 S. 33rd st., 203 Moore Building, Philadelphia, Pennsylvania 19104-6314, USA.}
\email[Corresponding author]{guraldan@seas.upenn.edu}

\author{Peter F. Stiller}
\address{Department of Mathematics, MS3368, Texas A\&M University, College Station, Texas 77843-3368, USA.}
\email{stiller@math.tamu.edu}

\maketitle

\begin{abstract} Following his discovery that finite metric spaces have injective envelopes naturally admitting a polyhedral structure, Isbell, in his pioneering work on injective metric spaces, attempted a characterization of cellular complexes admitting the structure of an injective metric space.  
A bit later, Mai and Tang confirmed Isbell's conjecture that a simplicial complex is injectively metrizable if and only if it is collapsible.
Considerable advances in the understanding, classification and applications of injective envelopes have since been made by Dress, Huber, Sturmfels and collaborators
, and most recently by Lang. Unfortunately a combination theory for injective polyhedra is still unavailable.

\smallskip
Here we expose a connection to the duality theory of cubings -- simply connected non-positively curved cubical complexes -- which provides a more principled and accessible approach to Mai and Tang's result, providing one with a powerful tool for systematic construction of locally-compact injective metric spaces:

\smallskip
{\bf Main Theorem. } Any complete pointed Gromov--Hausdorff limit of locally-finite piecewise-$\ell_\infty$ cubings is injective.\ep

\smallskip
This result may be construed as a combination theorem for the simplest injective polytopes, $\ell_\infty$-parallelopipeds, where the condition for retaining injectivity is the combinatorial non-positive curvature condition on the complex. Thus it represents a first step towards a more comprehensive combination theory for injective spaces. 

\smallskip
In addition to setting the earlier work on injectively metrizable complexes within its proper context of non-positively curved geometry, this paper is meant to provide the reader with a systematic review of the results~ ---~ otherwise scattered throughout the geometric group theory literature~ ---~ on the duality theory and the geometry of cubings, which make this connection possible.
\end{abstract}

\bigskip
\noindent {\bf Keywords:} Injective metric space, cubing, poc set, median algebra
\bigskip

\noindent {\bf 2010 MSC:} 51K05 (primary); 57M99, 05C10

\section{Introduction}
This paper arose out of an investigation of the mathematical foundations of the problem of unsupervised clustering of large data sets modeled as finite metric spaces.
In Section~\ref{clustering}, we describe the link between practical clustering and the theoretical work that follows.
Example ``real-world'' applications are community detection in large networks; semantic partitioning of point clouds; image segmentation and the derivation of phylogenetic trees, among many others.
Historically rooted in the field of phylogenetics (see~\cite{Felsenstein} for some history and discussion), much of the initial effort in this field was invested in obtaining consistent approximations of metric spaces by various ``treelike'' objects, such as metric trees and dendrograms~(see~\cite{Mullner-modern_clustering_algorithms} for an overview of relevant methods and~\cite{Carlsson_Memoli-stability,Carlsson_Memoli-classifying_clustering_schemes} for a more modern tack).
Most notably, a deep and principled approach formulated by Buneman~\cite{Buneman} has led to a powerful and prolific thrust by Dress, Sturmfels and collaborators (see~\cite{Sturmfels-biology_new_theorems,Dress_Huber_Moulton-clustering} for overviews) towards understanding a metric space in terms of canonically associated split-decomposable metrics~\cite{Bandelt_Dress-cut_metrics},~ ---~ one of several higher-dimensional generalizations of trees~ ---~ leading to methods for distance-based clustering with overlaps.
Subsequent work~\cite{DHM-Buneman_and_tight_span,DHM-envelope_cut_points,DMSW-splits_from_cuts_in_envelope} has made it clear that the {\em injective envelope}\footnote{Independently introduced in~\cite{Isbell-injective_envelope,Dress-tight_extensions} and~\cite{CL-one_more_envelope}, injective envelopes have since been referred to as {\em injective hulls} and {\em tight spans}. We will mostly refer to them simply as `envelopes,' for short.} of a metric space $(X,d)$ plays a fundamental, though not yet completely understood, role in distance-based clustering. 

\medskip
The injective envelope of a metic space $(X,d)$ may be seen as a complete geodesic extension satisfying certain minimality requirements. While it is relatively easy to show that the envelope of a finite tree-metric space is the geometric realization of an edge-weighted tree (the weights being interpreted as edge lengths), envelopes of more general finite metric spaces have a natural piecewise $\ell^\infty$ polyhedral structure of very high complexity. This has been studied in complete detail by Sturmfels and Yu for spaces with up to six points~\cite{Sturmfels_Yu-six_point_envelopes}. A toolkit for studying this structure for envelopes of more general discrete metric spaces was developed by Lang~\cite{Lang-injective_envelopes_and_groups}, with applications to group theory in mind, where injective envelopes are proposed as an alternative to existing polyhedral model spaces (such as the Rips complex) for some classes of finitely generated groups (e.g. Gromov-hyperbolic groups). We review the relevant notions Section~\ref{injective envelopes}, and expound on the connection between injectivity and clustering via tree-inspired splits in Section~\ref{splits}.

\medskip
In parallel with these developments, other useful abstractions of the graph-theoretic notion of a tree have emerged. In particular, Buneman's tree metric spaces, defined as metrics obtained from an edge-weighted tree by restricting to a finite subspace, may be seen as a special case of {\em cut metrics}, which, in turn, may be characterized as metrics induced on a finite subset from a median metric space (see last paragraph of Section~\ref{splits} for some details). The discrete variant of a median metric space is a piecewise-$\ell_1$ {\em cubing}~ ---~ a simply-connected non-positively curved cubical complex (see Definition~\ref{defn:cubing}) where the geometries of all cells are modeled on finite-dimensional axis-parallel parallelopipeds in $\ell_1$. This, in view of the relative simplicity of the cellular structure of a cubing as well as of results on the geometry of envelopes of cut metrics (again see the end of Section~\ref{splits}), motivates further study of natural relationship between cubings and injective metric spaces.

\medskip
The purpose of this paper is twofold. First, we provide an extension and a new, significantly simplified, proof of a result of Isbell~\cite{Isbell-injective_envelope} and Mai and Tang~\cite{Mai_Tang-collapsible_is_injective} on the existence of injective metrics on finite collapsible simplicial complexes by leveraging a connection between the geometry of piecewise-$\ell_1$ cubings and the geometry of the same cubings, but taken with a piecewise-$\ell_\infty$ geometry. An overview of the proof is provided in Section~\ref{results overview}. Second, this paper was intended to collect in one place the results~ ---~ otherwise scattered throughout the geometric group theory literature~ ---~ relating the geometry of cubings with their combinatorics, especially as expressed by the duality, discovered by Sageev~\cite{Sageev-thesis} and Roller \cite{Roller-duality}, between cubings and partially-ordered complemented sets (or {\em poc sets}, for short). Incidentally, putting these results in one place facilitates a rather self-contained description of the piecewise-$\ell_\infty$ geometry of cubings, appearing here for the first time.

\subsection{A motivating application: distance-based clustering}\label{clustering}
Loosely stated, the prototypical problem of distance-based clustering is that of consistently assigning a partition $\hsm{P}(X,d)$ of the base set $X$ to every finite metric space $(X,d)$; the elements of $\hsm{P}(X,d)$ are usually referred to as {\it clusters} of $(X,d)$ with respect to the particular clustering method. For the purposes of this discussion we will adopt the restriction {\it that the clustering method be a well-defined mapping} in the sense that, for every non-empty set $X$, the assignment $(X,d)\mapsto\hsm{P}(X,d)$ is, as the notation suggests, a function $\hsm{P}_X$ of the set $\metr{X}$ of all metrics\footnote{We refer to non-negative symmetric functions $d:X\times X\to\RR$ as `metrics,' if they satisfy the triangle inequality $d(x,y)\leq d(x,z)+d(z,y)$ for all $x,y,z\in X$. Contrary to more standard naming practices (e.g. ``pseudo-metric''), it is convenient in the context of finite metric spaces not to exclude from the definition functions possibly satisfying $d(x,y)=0$ for pairs $x,y$ with $x\neq y$, so that $\metr{X}$ becomes a {\em closed} pointed convex cone in the real vector space $\RR^{X\times X}$.} on $X$ to the set of all partitions of $X$. It is important to mention, though, that some overwhelmingly popular clustering methods do not satisfy this requirement. For example, {\it $K$-means clustering}\footnote{$K$-means clustering is defined for point-clouds in Euclidean space (that is, $X\subset\RR^n$, with the metric $d$ induced from the standard Euclidean structure), but could be applied to a general metric space $(X,d)$ after, say, a minimal distortion embedding.} is obtained through what is essentially a gradient-descent algorithm, known as Lloyd's algorithm\footnote{See~\cite{Lloyd} for the paper introducing K-means clustering, and~\cite{SWM-fast_large_Kmeans} for an example modern discussion of variations necessitated by current practical challenges which emerged with the onset of the era of  `big data' analysis.}, where the target function is known, in general, to have multiple local minima in the space of partitions. Consequently, the output of the algorithm is sensitive to the choice of seed partition provided at initialization.

\medskip
Phylogenetics motivates a slightly different, more general approach to clustering, called {\em hierarchical clustering}, realized formally in replacing the ranges of the clustering maps $\hsm{P}_X$ with more versatile spaces. For example, if one is interested in classifying a collection of individuals, represented by the points of $X$, at varying scales determined by the hypothetical moments in time when their ancestral lines diverged, the correct objects to map to are not mere partitions, but, rather, rooted metric trees whose leaves are bijectively labelled by the points of $X$. Consequently, one requires a mapping from $\metr{X}$ to the {\em space of phylogenetic trees} over $X$, the space of geometric realizations of rooted trees with leaf set $X$. A polyhedral model for this space admitting a CAT(0) geometry (see Section~\ref{NPC} below) was constructed by Billera, Holmes and Vogtmann~\cite{BHV-treespace}, who were motivated by the need for mathematical foundations for the statistical analysis of the output of hierarchical clustering maps. 

\medskip
In a different situation, one might be interested in extracting a simplified model of genealogical proximity from the sample space $(X,d)$, consequently trying to map $\metr{X}$ merely to a metric tree which contains an isometric copy of $X$, without enforcing an explicit representation of the common ancestor for all samples, or that the samples be represented as leaves. This allows, for example, for a sample to lie on the geodesic between two other samples, enabling inferences regarding its intermediacy. See~\cite{Dress_Huber_Moulton-clustering} for some history and a discussion of the mathematics of Phylogenetic analysis.

\medskip
Following the terminology of Dress~\cite{Dress-tight_extensions}, consider:
\begin{definition}\label{defn:tree,tree metric} A metric space $(Z,d)$ is said to be a tree\footnote{The definition of a tree provided here is, as stated, stronger than that of an $\RR$-tree, but may be shown to be equivalent to it~ ---~ see~\cite{Mayer_Oversteegen-Rtrees} for a detailed discussion.} if it is uniquely geodesic, and for any arc $z\brr [0,1]\to Z$, $t\mapsto z_t$ and for any $t\in[0,1]$ one has $d(z_0,z_t)+d(z_t,z_1)=d(z_0,z_1)$. A metric $d'$ on a finite set $X$ is said to be a {\em tree metric on $X$}, if $(X,d')$ embeds isometrically in some tree $(Z,d)$. 
\end{definition}
Buneman in~\cite{Buneman} proposes to study (and constructs) clustering maps of the form $\hsm{P}_X:\metr{X}\to\metr{X}$ which are non-expansive retractions onto the subspace of tree metrics on $X$.  Bunemann's construction proceeds as follows. First, given a metric $d$, one constructs a system $\mathscr{N}$ of nested binary partitions (or {\em splits}) $\sigma=\{S,X-S\}$ of $X$: every split $\sigma$ has a `width' parameter associated with it, 
\begin{displaymath}
	\mu^d_\sigma=\frac{1}{2}\min_{\substack{x,y\in S\\ z,w\notin S}}
	\left(
		\min\left(\begin{array}{c}
			d(x,z)+d(y,w)\\
			d(x,w)+d(y,z)
		\end{array}
		\right)-(d(x,y)+d(z,w))
	\right)
\end{displaymath}
and $\mathscr{N}$ is the set of splits of positive width. A tree metric is obtained by setting $\hsm{P}_X(d)=\sum_{\sigma\in\mathscr{N}}\mu^d_\sigma\delta_\sigma$, where $\delta_\sigma(x,y)\in\{0,1\}$ is zero if and only if both $x,y\in S$ or both $x,y\notin S$. The metrics $\delta_\sigma$ are called {\em cuts}. Non-negative combinations of cuts are called {\em cut metrics}, and Buneman characterizes tree metrics precisely as those cut-metrics which may be writtend down as non-negative combinations of cuts from a nested\footnote{The word used by Buneman is `compatible,' but here we will stick to the terminology that was developed for cubings, as it seems more evocative of the right geometric intuition.} system of splits. Similar notions of clustering mappings of this form are discussed in~\cite{Moulton_Steel-refined_Buneman}.

\medskip
An important thing to notice about Buneman's construction is not only that Buneman's clustering map produces a tree metric on $X$, but that it also produces an explicit combinatorial description of a tree $Z$ in which this tree metric embeds: in a nutshell, each $S\in\hsm{N}$ may be seen as the partition of $X$ induced by removing a single edge of $Z$ (more details about the notion of nesting appear below in Section~\ref{section:geometry of cubings}).

\subsection{Injective envelopes}\label{injective envelopes} The study of injective metric spaces arose from the study of hyper-convexity in functional analysis (see Theorem~\ref{injectivity criterion}). Most notably, the characterization of hyper-convexity by Aronszajn and Panitchpakdi led to Isbell's study of this class of spaces from a categorical viewpoint. Isbell introduces the category of metric spaces with non-expansive maps as morphisms, and considers the injective objects of this category with respect to the class of isometric embeddings:\footnote{Note that the requirement from $i\brr A\to B$ to be an isometric embedding rather than just a monic map in the category produces a notion of injectivity that is weaker than monic-injectivity.} 
\begin{definition}[Injective Metric Space \cite{Isbell-injective_envelope}]\label{defn:injective space} A metric space $X$ is said to be injective, if for any isometric embedding $i\brr A\to B$ and any non-expansive map $f\brr A\to X$ there exists a non-expansive $F\brr B\to X$ satisfying $F\circ i=f$. 
\end{definition}

An injective space $X$ is geodesic: pick $x,y\in X$ and consider the isometry $i\brr \{0,d(x,y)\}\to X$ with $i(0)=x$ and $i(d(x,y))=y$ and its mandated non-expansive extension to $F\brr [0,d(x,y)]\to X$; taking $0<s<t<d(x,y)$ and applying the triangle inequality twice easily leads to the conclusions that $d(x,F(s))=s$, $d(F(s),F(t))=t-s$ and $d(F(t),y)=d(x,y)-t$.

\medskip
An injective space $X$ is complete: let $\widetilde{X}$ denote the Cauchy completion of $X$; then the identity map from $X$ to $X$ extends to a non-expansive map of $\widetilde{X}$ to $X$; since $X$ is dense in $\widetilde{X}$ this map must be an isometry, so $X$ is complete.

\medskip
One of the simplest non-trivial examples of a class of injective spaces  is the class of trees introduced above in Definition~\ref{defn:tree,tree metric}. Lang provides a direct argument (see~\cite{Lang-injective_envelopes_and_groups}, Proposition 2.1), and an indirect one may be obtained through the equivalence between injectivity and hyper-convexity:
\begin{theorem}[Aronszajn-Panitchpakdi~\cite{hyperconvex_spaces}]\label{injectivity criterion} A metric space $(X,d)$ is injective if and only if it is hyper-convex: every finite collection of closed balls $\{B(p_i,r_i)\}_{i=1}^n$ in $X$ satisfying $r_i+r_j\geq d(p_i,p_j)$ for all $1\leq i,j\leq n$ has a common point.
\end{theorem}
An excellent and largely self-contained exposition of injectivity in metric spaces is provided in Section 2 of~\cite{Lang-injective_envelopes_and_groups}, including an independent proof of the above characterization, given there as Proposition 2.3.

\medskip
Isbell, in~\cite{Isbell-injective_envelope}, attempts two tasks that are natural in the categorical formulation of injectivity described above: the construction of minimal injective objects and the classification of injective objects. 
For the first task, Isbell proves the existence of injective envelopes. Recall the definition:
\begin{definition} Let $X$ be a metric space. An injective envelope for $X$ is an isometric embedding $e\brr X\to \epsilon X$ into an injective metric space $\epsilon X$ such that any isometry $f\brr X\to Z$ of $X$  into an injective metric space $Z$ may be written\footnote{In fact, Theorem 3.3 in~\cite{Lang-injective_envelopes_and_groups}, which discusses properties of Isbell's construction of an injective envelope for $X$,  implies that the map $g$ is uniquely determined by $f$. In the language of category theory, $\epsilon X$ is the result of a universal construction.} as $f=g\circ e$ for some isometry $g:\epsilon X\to Z$. 
\end{definition}
In fact, Isbell's construction, later independently rediscovered by Dress~\cite{Dress-tight_extensions}, is explicit enough to demonstrate that $\epsilon X$ is a compact polyhedron in $\ell_\infty(X)$ when $X$ is finite. This leads to a natural question, which is a part of the classification task of injective metric spaces:
\begin{question} Which polyhedra can be endowed with an injective metric?
\end{question}
Isbell shows that in order to support an injective metric, a (simplicial) polyhedron $X$ must satisfy some basic {\it topological} requirements. For example: $X$ needs to be collapsible.

\medskip
In search of a simple example, consider trees again. Starting from a non-empty finite set of points $X$ in a tree $(Z,d)$,~ ---~ recall Definition~\ref{defn:tree,tree metric}~ ---~ the union $T$ of all arcs in $Z$ joining points of $X$ is the geometric realization of a finite, edge-weighted, combinatorial tree. The argument provided earlier for demonstrating that injective spaces are geodesic may be extended (through the use of medians~ ---~ see paragraph preceding Proposition 2.1 in~\cite{Lang-injective_envelopes_and_groups}) to show that any isometric embedding of $X$ into an injective space $Y$ extends to an isometric embedding of $T$ into $Y$, hence $T$ is an injective envelope for $X$.

\medskip
Tying this example back to distance-based clustering is the fact that the metric $d$ of this example is a tree metric from the start. If $\hsm{P}_X$ is taken to be Buneman's retraction, then $\hsm{P}_X(d)=d$, and we may take $Z$ to be the tree constructed from the splits provided by Buneman's construction. We have just seen that the injective envelope $T$ of $(X,d)$ is contained in $Z$ (in the sense of being isometrically embeddable in $Z$), and, from the fact (to be seen later) that every edge of $Z$ must separate a pair of points in $X$ one deduces that $T=Z$. In other words, for any metric $d$ on $X$, Buneman's construction recovers the injective envelope of $\hsm{P}_X(d)$ (in the form of an edge-weighted combinatorial tree!).

\subsection{The geometry of splits}\label{splits} 
In the context of the clustering problem, the injective envelope serves as a tool for transforming a finite collection of disparate points~ ---~ the data $(X,d)$~ ---~ into a contractible space whose connectivity properties may be studied through, for example, mappings to trees. Intuitively, the injective envelope $\epsilon(X,d)$ is ideal in its role as a filling: it is the `leanest' extension among all `freest' extensions of $(X,d)$. The `freedom' we refer to here is geodesics in injective metric spaces being minimally constrained (hyper-convexity states, so to speak, that you get at least as many geodesics as you need to efficiently connect any number of points through a single commuting station), while `leanness' is to be understood in the sense of the envelope embedding isometrically into {\em any} injective space.

\medskip
Formally, given a finite metric space $(X,d)$, let $\bar d=\hsm{P}_X(d)$ be the Buneman projection of $d$ to the space of tree metrics, and let $T$ be the tree recovered from $\bar d$ as described above, which is also the injective envelope of $\bar{d}$. Since the identity mapping from $(X,d)$ to $(X,\bar{d})$ is non-expansive, it extends to a non-expansive mapping $f$ of $\epsilon(X,d)$ onto $T$, because $T$ is injective. An edge $e$ of $T$ of length $\mu$ pulls back to a cut set $C=f\inv(e)\subset\epsilon(X,d)$ splitting $\epsilon(X,d)$ into two subspaces $A,B$ such that (1) $X\subset A\cup B$, and (2) $d(a,b)\geq\mu$ for all $a\in A$, $b\in B$. From here, one can use these structures in the clustering process, for example: one could ask what partition of $X$ is obtained through the removal of all edges of $T$ of length greater than or equal to some threshold $\delta>0$.

\medskip
For another example illustrating how {\em topological} features of envelopes may be relevant to clustering, recall the work of Ward~\cite{Ward-axioms_for_cutpoints}, which deduces a treelike quotient of a connected and locally connected compact Hausdorff space from its set of cut points. Using~\cite{DHM-envelope_cut_points} to compute the cut points of the injective envelope $\epsilon(X,d)$, one uses the fact that $\epsilon(X,d)$ is a finite polyhedron to argue that the Ward quotient is a finite tree $T$, which, similarly to Buneman's tree, could be used for clustering.

\medskip
It was a fundamental observation of Bandelt and Dress that one need not restrict attention to nested split systems (and hence to trees). Relaxing the notion of width for a split $\sigma=\{S,X-S\}$ from $\mu^d_\sigma>0$ to $\alpha^d_\sigma>0$, where
\begin{displaymath}
	\alpha^d_\sigma=\frac{1}{2}\min_{\substack{x,y\in S\\ z,w\notin S}}
	\left(
		\max\left(\begin{array}{c}
			d(x,z)+d(y,w)\\
			d(x,w)+d(y,z)
		\end{array}
		\right)-(d(x,y)+d(z,w))
	\right)
\end{displaymath}
they prove (see~\cite{Bandelt_Dress-cut_metrics}, Theorems 2 and 3) that every metric $d$ may be written as $d=d_0+d_{\mathscr{S}}$, where (1) $d_0$ has no split $\tau$ with $\alpha^{d_0}_\tau>0$ (that is, $d_0$ is {\em split-prime}); (2) $d_{\mathscr{S}}=\sum_{\sigma\in\mathscr{S}}\alpha^d_\sigma\delta_\sigma$ (the {\em totally split decomposable part} of $d$), and (3) the family $\mathscr{S}$ of splits $\sigma$ satisfying $\alpha^d_\sigma>0$ must satisfy a combinatorial condition called {\em weak compatibility}. Moreover, for any family of splits $\mathscr{S}$, setting $d=d_{\mathscr{S}}$ as above yields $d_0=0$ if and only if $\mathscr{S}$ is weakly compatible.  

\medskip
Thus, a new projection~ ---~ this time onto a space of cut metrics properly extending the space of tree metrics~ ---~ is obtained by mapping $d\mapsto d-d_0$. This projection is more informative than Buneman's in the sense that the containment $\mathscr{N}\subseteq\mathscr{S}$ implies that Buneman's projection factors through this one. This gives rise to a non-expansive mapping from $\epsilon(X,d)$ to that of $\epsilon(X,d-d_0)$ in the same manner as before, encouraging questions regarding ways to characterize the family $\mathscr{S}$ of splits of $X$ by cut sets in $\epsilon(X,d)$.

\medskip
For totally decomposable metrics $d=d_{\mathscr{S}}$, much work has been done studying their injective envelopes in the series of papers~\cite{DHM-Buneman_and_tight_span,DHM-median_vs_tight_span,DHM-more_Bunemann_vs_tight_span}, with emphasis on the role of the {\em Buneman complex}, introduced in~\cite{DHM-the_Buneman_complex}, associated with the split system $\mathscr{S}$ (Theorem 3.1 of~\cite{DHM-median_vs_tight_span} is a good main result to keep in mind). It is now clear, following the independent work of Roller~\cite{Roller-duality}, that the Buneman complex is a geometric realization of the cubing dual to the split system $\mathscr{S}$.

\medskip
Finally, it ought to be mentioned that not all cut metrics are totally split decomposable. In fact, metrics as simple as the Hamming metric on $X=\{0,1\}^m$, $m\geq 3$ are split-prime, which means they are collapsed to points (that is, spaces of zero diameter) under the Bandelt--Dress projection. Nevertheless, recalling that $(X,d)$ is a cut metric if and only if it embeds isometrically in an $\ell_1$-space (Theorem 4.2.6 in~\cite{Deza_Laurent-cutbook}) puts us in a position to also consider $(X,d)$ as a candidate for embedding in a median metric space, the continuous analog of a cubing serving as a higher-dimensional notion of a metric tree (mainly in view of Corollary 5.4 in~\cite{Chatterji_Drutu_Haglund-measured_spaces_with_walls}, where {\em measured spaces with walls}~ ---~ a vast generalization of cut metrics~ ---~ are introduced). Thus, one might hope that the direct connection between the Buneman complex of a totally split-decomposable metric (which is a cubing) and the injective envelope of that metric seems to be only a special case of a more general theory relating splits in $(X,d)$ with measured wall spaces on $X$, with canonically defined cuts in $\epsilon(X,d)$.

\subsection{A little bit on non-positive curvature}\label{NPC}
Although CAT(0) geometry does not play a direct role in this work, it has been (and will further be) mentioned as a source of motivation in this text sufficiently to merit a brief review of the relevant notions. For a much more detailed review we refer the reader to~\cite{Bridson_Haefliger-nonpositive_curvature}, chapter II, on which ours is based.

\medskip
Recall that a {\em geodesic triangle} $\Delta$ with vertices $x,y,z$ in a metric space $(X,d)$ is the union of geodesic arcs $[x,y]$, $[y,z]$ and $[x,z]$,~ ---~ the sides\footnote{As the space $(X,d)$ may not be uniquely geodesic, the notation $[x,y]$ only comes to indicate a particular choice of a geodesic arc joining the endpoints $x$ and $y$.} of the triangle~ ---~ and that a {\em comparison triangle} for $\Delta$ in the Euclidean plane $\EE^2=(\RR^2,d_2)$, $d_2(x,y)=\left\Vert x-y\right\Vert_2$ is a geodesic triangle $\bar\Delta$ in $\EE^2$ with vertices $\bar{x},\bar{y},\bar{z}$ such that $d_2(\bar{p},\bar{q})=d(p,q)$ for $p,q\in\{x,y,z\}$. In other words, each side of $\Delta$ may be mapped isometrically onto the corresponding side of $\bar\Delta$. If $\bar\Delta$ in $\EE^2$ is a comparison triangle for a geodesic triangle $\Delta$ in $(X,d)$, then every point $p\in\Delta$ has a uniquely defined {\em comparison point}, denote $\bar{p}\in\bar\Delta$: simply find a side of $\Delta$ containing $p$ and map it to the corresponding side of $\bar\Delta$; the image of $p$ under this mapping is the desired point $\bar{p}$. Note that a comparison triangle in $\EE^2$ always exists, and is unique up to Euclidean isometry, which makes the following definition meaningful:

\begin{definition}\label{defn:CAT0} A geodesic triangle $\Delta$ in a metric space $(X,d)$ is said to satisfy the CAT(0) inequality if for every $p,q\in\Delta$ one has $d(p,q)\leq d_2(\bar{p},\bar{q})$.
A metric space is said to be CAT(0), if it is geodesic, and it satisfies the CAT(0) inequality.
\end{definition}

Multiple characterizations of the CAT(0) inequality exist (\cite{Bridson_Haefliger-nonpositive_curvature}, chapter II.1). Of the most important properties of CAT(0) spaces one should probably mention the following: a CAT(0) space is uniquely geodesic and contractible ({\it loc. cit.}, proposition II.1.4); every closed convex subset $\varnothing\neq C\subset X$ has a well-defined, non-expansive closest point projection $\pi_C:X\to C$ which is also the endpoint of a strong deformation retraction of $X$ onto $C$ ({\it loc. cit.}, proposition II.2.4); every bounded set has a center ({\it loc. cit.}, proposition II.2.7).

\medskip
A local version of the CAT(0) inequality is as important as the global notion:
\begin{definition} A metric space $(X,d)$ is said to be non-positively curved in the sense of Alexandrov, if every $x\in X$ has some $r_x>0$ such that the open ball $B_d(x,r_x)$ is CAT(0).
\end{definition}
Non-positively curved spaces are relatively easy to construct as finite polyhedra whose cells are chosen to be isometric to polytopes in Euclidean or hyperbolic space: one needs to make sure that the geometric links (see~\cite{Bridson_Haefliger-nonpositive_curvature}, I.7.14-18) satisfy the CAT(1) inequality (meaning that all geodesic triangles of perimeter less than $2\pi$ satisfy the inequality of Definition~\ref{defn:CAT0} with respect to their comparison triangles on the standard sphere of unit curvature, $\SS^2$). This observation was made by Gromov in~\cite{Gromov-hyperbolic_groups}, where he also proved a version of the Cartan--Hadamard theorem (\cite{Bridson_Haefliger-nonpositive_curvature}, II.4.1) guaranteeing that the universal cover of a non-positively curved space is CAT(0). Thus, CAT(0) spaces may be constructed as universal covers of finite non-positively curved piecewise-Euclidean/hyperbolic polyhedra, with Bridson's theorem on shapes (\cite{Bridson_Haefliger-nonpositive_curvature}, Theorem I.7.50) guaranteeing their completeness and the existence of geodesics.

\medskip
A particular family of interest to Gromov in~\cite{Gromov-hyperbolic_groups} was the family of {\em cubical complexes}, where it is required that all cells are embedded Euclidean cubes, glued together by isometries among their faces. He observed that the link of every vertex in such a complex is a simplicial complex, and concluded that the CAT(1) inequality for geometric links is obtained if and only if all the vertex links in the complex are flag complexes (see~\cite{Bridson_Haefliger-nonpositive_curvature}, II.5.15-20 for details in the finite-dimensional case; the general case is due to Leary~\cite{Leary}, Theorem B.8). Hence the definition of a cubing:
\begin{definition}\label{defn:cubing} A cubed complex $X$ is said to be non-positively curved (NPC) if the link of each vertex in $X$ is a simplicial flag complex. The complex $X$ is said to be a {\it cubing}, if it is non-positively curved and simply-connected.
\end{definition}
A few words are in order regarding our insistence on using the term `cubing' rather than ``CAT(0) cubical complex''. Apart from our intention to use the same underlying combinatorial structure to support a metric that is patently not CAT(0), the emphasis on the combinatorics derives from basic questions regarding necessary properties of the piecewise-Euclidean metric on a general cubing. A cubing, in general, may not have a compact quotient by a properly-discontinuous group of cellular maps (e.g., when the dimension of cubes in the complex is unbounded), in which case it will not arise as a result of the construction described above. As stated, Bridson's theorem on shapes does apply to all finite-dimensional cubings, so the problem in the general case is that it is not immediately clear that the piecewise-Euclidean metric on a general cubing makes it into a complete geodesic metric space (though when it does, the Cartan--Hadamrd theorem and Gromov's NPC criterion do guarantee the CAT(0) property). For more information about the general case, see~\cite{Leary}, Appendices A-C.

\subsection{Our results}\label{results overview}
We now return to the result by Isbell (in dimension $\leq 2$) and Mai and Tang (in higher dimensions, stating that any collapsible simplicial complex admits an injective metric. Two components of Isbell's argument hint to a deeper connection to (global) non-positive curvature in the sense of Alexandrov:
\begin{itemize}
	\item Gluing a pair of injective spaces along a point results in an injective space (providing yet another way to verify that the geometric realization of a finite edge-weighted combinatorial tree is injective). This is, of course, a far cry from the result that gluing CAT(0) spaces along convex subspaces yields a CAT(0) space (see~\cite{Bridson_Haefliger-nonpositive_curvature}, Section II.11), leaving much to be desired in a combination theory for injective spaces.
	\item Isbell's construction of an injective metric on a collapsible 2-dimensional simplicial complex makes explicit use of combinatorial non-positive curvature conditions (see Definition~\ref{defn:cubing}) appearing in a refined decomposition of the complex into squares.
\end{itemize}

Let us now study Isbell's argument in more detail as we analyze the connection with non-positive curvature and introduce our own results.

\medskip
In \cite{Isbell-injective_envelope}, Isbell proves that a (finite) collapsible 2-dimensional cellular complex $X$ admits an injective metric by explicitly constructing a hyper-convex metric on $X$ as follows: taking a triangulation of $X$, he subdivides its triangles into squares so as to form what he calls a {\it collapsible cubical 2-complex}, $\Delta$. He then metrizes $X$ as a geometric realization of $\Delta$, having first realized each 2-cube as a copy of the unit cube in $(\RR^2,\Vert\cdot\Vert_\infty)$ and endowing the resulting 2-dimensional piecewise-$\ell_\infty$ polyhedron with the associated quotient metric. Mai and Tang's proof of Isbell's conjecture \cite{Mai_Tang-collapsible_is_injective} extends this construction to higher dimensions. The verification of injectivity then proceeds by verifying the intersection property for finite families of balls $\{B(p_i,r_i)\}$ stated in Theorem~\ref{injectivity criterion}, in two steps:
\begin{enumerate}
	\item Reduction to the case where all the $p_i$ are vertices of the cubical subdivision and all the radii $r_i$ are integers;  
	\item Applying the properties of a ``collapsible cubical 2-complex'' to verify the result. 
\end{enumerate}
Among Isbell's requirements of a collapsible cubical 2-complexes one immediately notices Gromov's `no-triangle' condition for non-positively curved cubical complexes. Indeed, upon closer inspection it becomes clear that the notion of a collapsible cubical 2-complex is exactly a 2-dimensional cubing in the language of modern geometric group theory. The analysis by Mai and Tang is far more opaque, because their argument proceeds by a rather technical induction argument on the dimension of the given polyhedron. As we demonstrate in this article, the modern outlook on non-positive curvature allows one to sweep the ``gory details'' under the rug of Sageev--Roller duality, leaving a neatly organized picture which is uniform in all dimensions. 

\medskip
We now proceed to outline our approach. Given a finite cubing, let its cubes be metrized as axis-parallel parallelopipeds in $(\RR^n,\Vert\cdot\Vert_\infty)$, where $n$ may vary. The edge lengths of the parallelopipeds may be chosen with some degree of freedom, subject to the gluing constraints of the complex. We call the resulting geodesic spaces {\it piecewise-$\ell_\infty$ cubings}. Our central result is:
\begin{theorem}\label{main theorem} Every finite piecewise-$\ell_\infty$ cubing is injective.
\end{theorem}
Thus, not only is it true that any finite combinatorial cubing is injectively metrizable (Mai and Tang \cite{Mai_Tang-collapsible_is_injective}), but, in fact, it carries a whole deformation space of injective metrics. Moreover, observing that the class of injective metric spaces is closed under pointed Gromov--Hausdorff limits (see lemma below) extends the scope of the above theorem to give the main result stated in the abstract, as a locally finite cubing may be exhausted by finite ones. One needs to exercise care, however, either to guarantee the completeness of the given cubing itself, or to pass to its completion (which, recall, is a necessary condition for injectivity).
\begin{lemma}[``Limit Lemma'']\label{limit lemma} A complete metric space arising as a pointed Gromov--Hausdorff limit of proper injective metric spaces is itself injective. 
\end{lemma}
Recall that a metric space is said to be proper if closed bounded subsets thereof are compact.

\medskip
The same lemma plays a crucial role in our reduction of the general case to the case of finite {\em unit} cubings. Returning to $X$ being a finite $\ell_\infty$-cubing, we explain how to see $X$ as a pointed Gromov--Hausdorff limit of a sequence of the form $\left(\tfrac{1}{n}X^{(n)},v\right)_{n\in\NN}$ where $X^{(n)}$ is a cubical subdivision of $X$ obtained by cutting the cubes of $X$ in a grid-like fashion (parallel to their faces), with unit weights. Applying the limit lemma once again we see that it now suffices to prove one final lemma. 
\begin{proposition}\label{final lemma} Every finite unit piecewise-$\ell_\infty$ cubing is injective. 
\end{proposition}
This is, essentially, the original result proved by Mai and Tang in \cite{Mai_Tang-collapsible_is_injective}, though in different language, and with some heavy lifting (the notion of a cubing was absent at the time). In this paper, we propose an alternative proof using the full power of the structure theory of cubings. Roughly speaking, the idea is to prove that the balls in the zero-skeleton $X^0$ of a unit piecewise-$\ell_\infty$ cubing $X$ are convex subsets of $X^0$ with respect to its unit piecewise-$\ell_1$ metric. Once this is known, the same property is inherited by all piecewise-$\ell_\infty$ cubings via the limit lemma. This finishes the proof of injectivity: as cubings are known to be geodesic median spaces (see below) with respect to their piecewise-$\ell_1$ metric, they satisfy a $1$-dimensional Helly theorem~ ---~ every finite collection of pairwise-intersecting convex sets has a common point. In particular, any collection of pairwise-intersecting $\ell_\infty$-balls in a cubing must have a common point, as desired.
 
\subsection{Remaining Questions} The class of complete piecewise-$\ell_\infty$ cubings does not coincide with the class of injective metric spaces. This follows directly from the limit lemma and the example below (see ex. \ref{example:three-fin} and fig. \ref{fig:three-fin}). It seems that neither does the slightly broader class of spaces arising as completions of locally finite piecewise-$\ell_\infty$ cubings (see discussion in Section~\ref{summability}). It would be interesting to quantify the discrepancy, perhaps in terms of the classification by Lang \cite{Lang-injective_envelopes_and_groups}.

\begin{figure}[t]
	\begin{center}
		\includegraphics[width=.6\columnwidth]{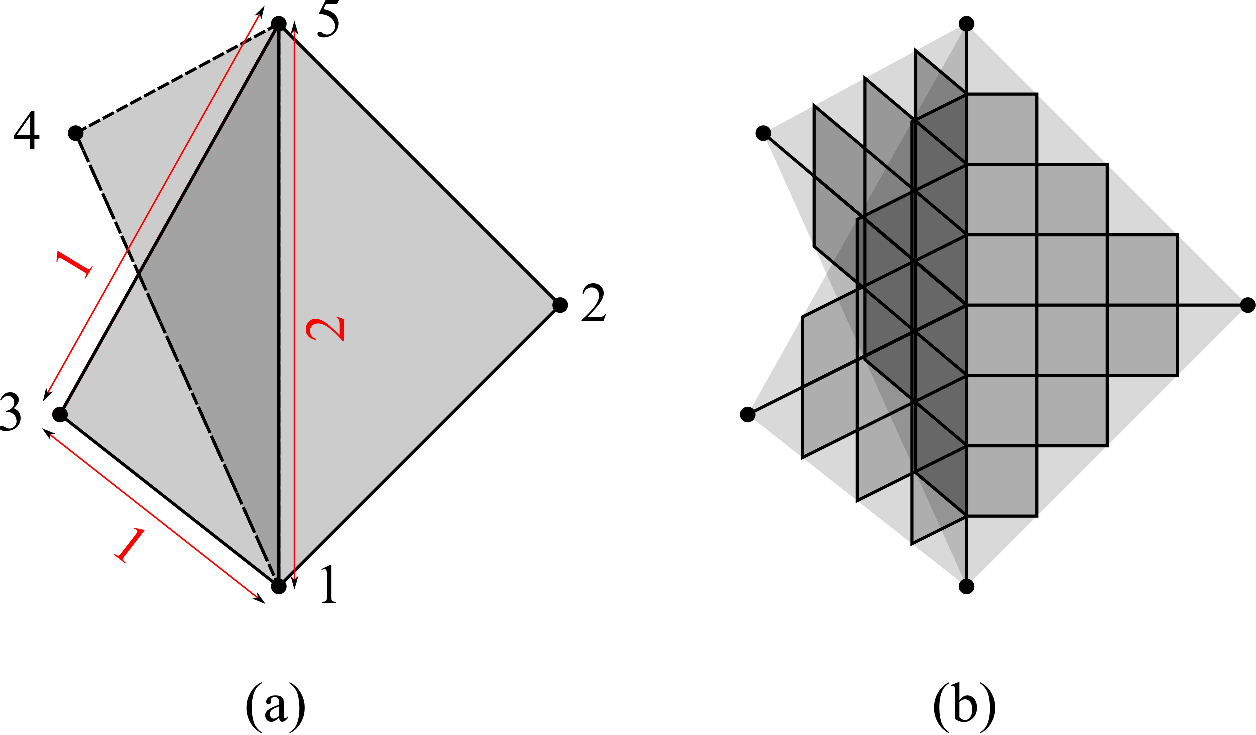}
	\end{center}
	\caption{The injective envelope of the 5-point space in Example~\ref{example:three-fin} (a) may be obtained as a limit of $\ell_\infty$-cubings of the form (b). We fondly refer to it as the ``3-fin''.
	\label{fig:three-fin}}
\end{figure}

\begin{example}[Injective Envelope of 5 Points]\label{example:three-fin} It is shown in \cite{Bandelt_Dress-cut_metrics} that the metric space $X=\{1,2,3,4,5\}$ with
\begin{eqnarray*}
	&&\dist{1}{2}=\dist{1}{3}=\dist{1}{4}=1,\\
	&&\dist{5}{2}=\dist{5}{3}=\dist{5}{4}=1,\\
	&&\dist{1}{5}=\dist{2}{3}=\dist{2}{4}=\dist{3}{4}=2,
\end{eqnarray*}
has the injective envelope depicted in Figure~\ref{fig:three-fin}(a), where one should think of the triple fin depicted there as the result of gluing three unit squares cut out from the $\ell_1$ plane and glued together to overlap along the (filled-in) triangles with sides $[1,5]$, $[1,x]$ and $[x,5]$ for $x\in\{2,3,4\}$. Each of these triangular fins is, in fact the limit of a sequence of piecewise-$\ell_\infty$ cubings, resulting in a sequence of approximations for $\epsilon X$ of the form shown in Figure~\ref{fig:three-fin}(b). 
\end{example}

While it is very possible that the class of limits of piecewise-$\ell_\infty$ cubings is still too narrow to exhaust all injective metric spaces, our results seem to suggest that the combinatorial structure of a cubing is nothing more than a set of explicit gluing instructions following which one could create a `big' injective space out of small, standardized pieces, namely: $\ell_\infty$-cubes. Thus, we would like to hope that our results are merely a glimpse of a combination theory for constructing `big' injective spaces out of `small'/`simple' ones. This motivates the following question. 
\begin{question} Is there a combination theory for injective metric spaces? If two injective spaces are glued along a convex (injective?) subspace, when is the resulting space injective?
\end{question}
A well-developed combination theory should simplify the proofs of the existing results as well as contribute to the understanding of the problem of characterizing injective spaces in constructive terms.

\section{Preliminaries}\label{section:preliminaries}
We use this section to recall some of the language required for the rigorous development of the main result. Some of the facts presented in this section seem to be common knowledge, yet new in the sense that they are not easily found in the literature --- we thought it better to include them here due to their elementary nature, as well as for the sake of providing a self-contained exposition.

\subsection{Piecewise-(your favorite geometry here) Cubed Complexes}\label{section:geometric complexes} Fix $p\in[1, \infty]$. The purpose of this section is to recall the necessary technical language for dealing with geometric cubed complexes modeled on $\ell_p$ geometry using the language and methods of \cite{Bridson_Haefliger-nonpositive_curvature}, Section I.7. 

A geometric cubical complex $X$ modeled on $\ell_p$ is obtained as a quotient of a disjoint union $\tilde X$ of a collection $\mathscr{S}$ of $\ell_p$ cubes. Each cube $S\in\mathscr{S}$ is a copy of $[0,1]^n\subset\RR^n$ for some $n\in\NN$, endowed with the metric $d_S$ induced by the $\ell_p$ norm. In complete analogy with simplicial complexes, $\mathscr{S}$ is required to include, for each cube $S\in\mathscr{S}$, all the faces of $S$ (which are also cubes). We endow $\tilde X$ with the metric $\tilde d(x,y)=d_S(x,y)$ if $x$ and $y$ share a cube in $\mathscr{S}$ and $d(x,y)=\infty$ otherwise. Subjecting $\tilde X$ to isometric (and hence also affine) identifications among some of its faces gives rise to a quotient map $\pi:\tilde X\to X$, with the restriction that $\pi\res{S}:S\to X$ is injective (compare with {\it loc.~cit.}, Definition 7.2). We will refer to such $X$ as {\em unit piecewise-$\ell_p$ cubical complexes}. The images $\pi(S)$, $S\in\mathscr{S}$ will be referred to as the {\em faces} of $X$, and we will write $\pi(S)<X$; since all identifications made are isometric, each $\pi(S)<X$ carries a well-defined metric, also denoted here by $d_S$, obtained by pushing $d_S$ forward along $\pi$. More generally, we allow a slight variation on this construction,~ ---~ called simply {\em piecewise-$\ell_p$ cubical complexes}~ ---~ which is achieved by putting non-negative real-valued weights on the coordinate axes of the individual cubes. One needs to make sure that the weights match, in the sense that any two cubes sharing a common face in $X$ do have their axes weighted in a way that keeps any pair of identified faces isometric to each other via the specified identification maps.

\medskip
We endow all such $X$ with the {\em quotient pseudo-metric}, which, in this situation may be described as follows (compare with {\it loc.~cit.,} Definition I.5.19):
\begin{definition}[Quotient Pseudo-Metric on a piecewise-$\ell_p$ cube complex]\label{quotient pseudo-metric} Let $\tilde X$, $X$ and $\pi$ be as above. Then the {\em quotient distance $d(x,y)$ between $x,y\in X$} is defined as the greatest lower bound on expressions of the form
\begin{equation}
	\sum_{i=0}^n \tilde d(x_i,y_i)
\end{equation}
where $x_0\in \pi\inv(x)$, $y_n\in \pi\inv(y)$ and $\pi(y_{i-1})=\pi(x_i)$ for all $i=1,\ldots,n$.
\end{definition}
As we shall see in Section~\ref{section:weighted cubings}, Roller's duality theory between cubings and discrete poc sets is ideally suited for the purpose of maintaining all the necessary records.

\medskip
One can further adapt the construction of the quotient pseudo-metric on a piecewise-$\ell_p$ cubical complex to its rather specialized strucure (compare {\it loc.~cit.}, Definition I.7.4):
\begin{definition}[strings, length] An {\em $m$-string from $x$ to $y$} in a piecewise-$\ell_p$ cubical complex $X$ is a sequence of points $\mbf{p}=(x_0,\ldots,x_m)$ where $x_0=x$, $x_m=y$ and every consecutive pair of points $x_{i-1},x_i$ is contained in a common face $S_i<X$. The {\it length} of $\mbf{p}$ is defined to be:
\begin{equation}
	\length{\mbf{p}}:=\sum_{i=1}^n d_{S_i}(x_{i-1},x_i)\,.
\end{equation}
A {\em string} is an $m$-string for some $m\in\NN$. The set of all strings from $x$ to $y$ in $X$ will be denoted by $\paths{x,y}$. 
\end{definition}
Since each face of $X$ is a geodesic metric space, it is easy to verify that the quotient pseudo-metric on $X$ between two points coincides with the greatest lower bound on the length of a string joining them ({\it verbatim} repetition of {\it loc.~cit.}, Lemma I.7.5). Moreover, there are some obvious optimizations to this picture.
\begin{definition}[taut strings]
A string $\mbf{p}$ in a piecewise-$\ell_p$ cubical complex is said to be {\it taut} if no consecutive triple of points along $\mbf{p}$ is contained in a cube.
\end{definition}
Since the individual cubes in $\tilde X$ are geodesic metric spaces, tightening a string locally will never increase its length. This leads to the following formula for the quotient pseudo-metric:
\begin{equation}\label{ell-p formula}
	d(x,y)=\inf\set{\length{\mbf{p}}}{\mbf{p}\in\paths{x,y}\text{ is taut}}.
\end{equation} 
\begin{figure}[t]
	\begin{center}
		\includegraphics[width=.6\textwidth]{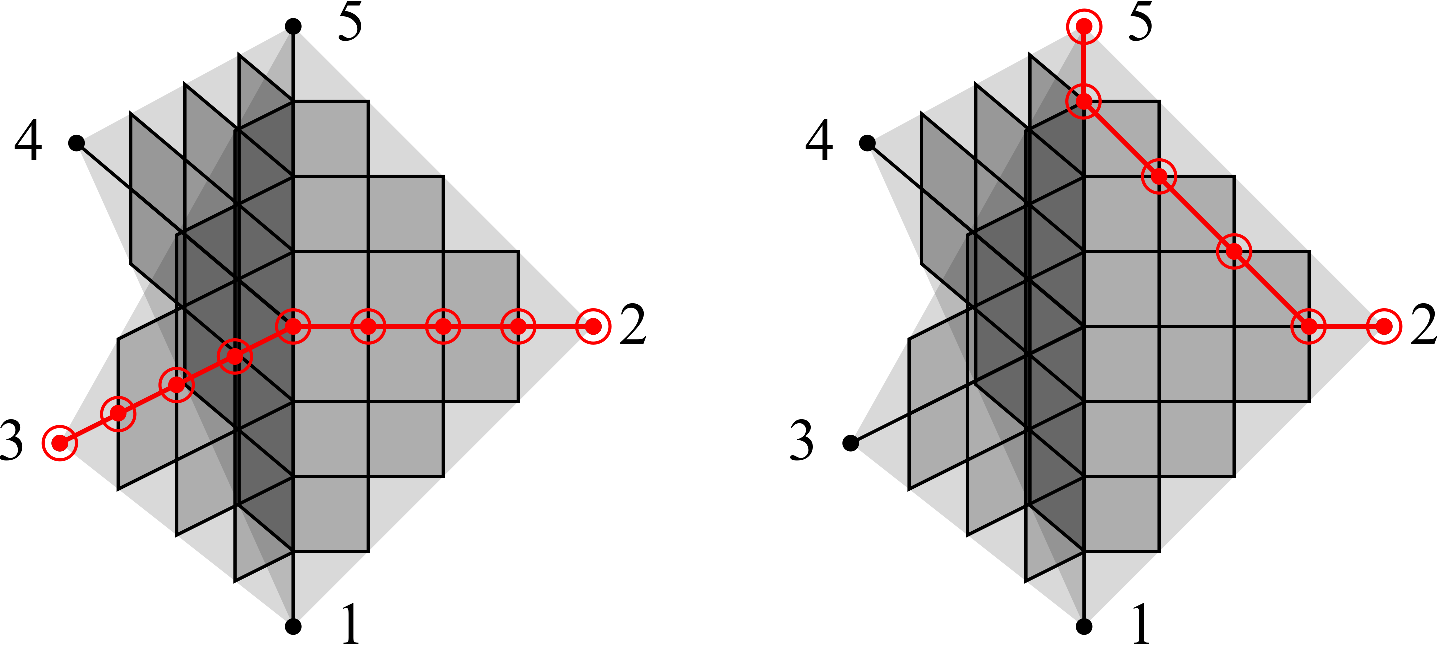}
	\end{center}
	\caption{Piecewise-$\ell_\infty$ cubical approximation of the $3$-fin (see Examples~\ref{example:three-fin},\ref{example:three-fin-paths}), the five-point space $X=\{1,2,3,4,5\}$ whose injective envelope is drawn in gray. In red are minimum length taut strings joining the points $2$ and $3$ (left), and $2$ and $5$ (right), where all weight values are equal to $\tfrac{1}{4}$.\label{fig:three-fin-paths}}
\end{figure}
Let us illustrate these definitions by taking the reader back to the $3$-fin, the injective envelope of the $5$-point metric space introduced in Example~\ref{example:three-fin}:
\begin{example}\label{example:three-fin-paths} If we metrize each cube in Figure~\ref{fig:three-fin}(b) as an axes-aligned cube in $\ell_{\infty}$ and weight the walls by $\tfrac{1}{4}$, we can calculate the distances in the resulting piecewise-$\ell_{\infty}$ cubing by finding an optimal taut string. Figure~\ref{fig:three-fin-paths} illustrates two types of such strings (from $2$ to $3$ and from $2$ to $5$) which produce the approximations $\dist{2}{3} = 2$ and $\dist{2}{5} = \frac{5}{4}$ of the original metric. Imagining finer cubical `approximations' of the 3-fin in the same spirit, we can see that, as the weights on the cubing tend to zero (while the number of cubes grows), the mapping of the five point metric space in Example~\ref{example:three-fin} to the approximating piecewise $\ell_\infty$ cubing approaches an isometric embeddeding into the limit space. 
\end{example}
The reduction of distances to lengths of taut strings raises the question of how easy it is to compute, say, geodesic paths for piecewise-$\ell_p$ cubings. We postpone this discussion until we have a better way of representing piecewise-$\ell_p$ cubings, in Section~\ref{section:weighted cubings}, and return to more basic questions: How far are we from knowing whether or not $X$ is a geodesic metric space? Whether or not $X$ is complete?

A necessary step along the way is to verify that the quotient pseudo-metric $d$ is, in fact, a metric on $X$, excluding some pathologies. Altering \cite{Bridson_Haefliger-nonpositive_curvature}, Definition I.7.8 slightly one defines, for each $x\in X$, the quantity
\begin{displaymath}
	\epsilon(x):=\inf_{x\in S<X}\epsilon(x,S)\,,\text{ where }
	\epsilon(x,S):=\inf\{d_S(x,T)\,|\,T<X, T\subset S, x\notin T\}\,.
\end{displaymath}
Then, the reasoning of {\it loc.~cit.}, I.7.9-13 applies {\it verbatim} (including the examples of pathologies) leading to two conclusions:
\begin{enumerate}
	\item If $\epsilon(x)>0$ for all $x\in X$, then $d$ is a length metric on $X$ ({\it loc. cit.}, Corollary I.7.10);
	\item If $X$ has only finitely many isometry classes of faces, then $(X,d)$ is a complete length space ({\it loc. cit.}, Theorem I.7.13).
\end{enumerate}
These results make it possible to immediately apply the Hopf--Rinow theorem ({\it loc.~cit.}, I.3.7) in the case of any finite piecewise-$\ell_p$ {\it cubical complex} to conclude that it is a complete geodesic metric space. An extension of this result to the case of a locally finite {\em unit} piecewise-$\ell_p$ {\it cubing} is then made possible, too, by constructing an exhaustion of the cubing by finite convex sub-complexes. Such an exhaustion is made possible by Theorem B.4 of~\cite{Leary}, which states that any finite sub-complex of a cubing is contained in a finite convex sub-complex. That theorem is a consequence of the properties of cubings as median spaces, and we will return to it in Corollary~\ref{finite convex hull}. For the more general case of a locally finite piecewise-$\ell_p$ cubing that is not necessarily unit, one gets to keep the existence of geodesics (because of there being an exhaustion by finite cubings), but not completeness (see Example~\ref{example:incomplete} in Section~\ref{summability}).

\subsection{Gromov--Hausdorff limits}\label{section:GH limits}
We refer the reader to Chapter I.5 of \cite{Bridson_Haefliger-nonpositive_curvature} for a detailed introduction and discussion of Gromov--Hausdorff convergence of proper metric spaces. Recall that a binary relation $R\subset X\times Y$ has {\it projections}
\begin{equation}
\begin{split}
	\pi_X(R)=\set{x\in X}{\exists y\in Y\;(x,y)\in R}\\
	\pi_Y(R)=\set{y\in Y}{\exists x\in X\;(x,y)\in R}
\end{split}
\end{equation}
and that a subset $Z$ of a metric space $X$ is said to be {\it  $\epsilon$-dense} (in $X$), if the collection of closed balls of radius $\epsilon$ about the points of $Z$ cover $X$. We will use the following convergence criterion for our technical work.
\begin{definition} Let $X,Y$ be metric spaces and let $\epsilon>0$. A relation $R\subseteq X\times Y$ is said to be an {\it $\epsilon$-approximation} between $X$ and $Y$, if:
\begin{enumerate}
	\item the projections $\pi_X(R)$ and $\pi_Y(R)$ are $\epsilon$-dense;
	\item $\left|\dist{x}{x'}-\dist{y}{y'}\right|<\epsilon$ holds for all $(x,y), (x',y') \in R$.
\end{enumerate}
An $\epsilon$-approximation is {\it surjective}, if $\pi_X(R)=X$ and $\pi_Y(R)=Y$.
\end{definition}
One of the ways to define the The Gromov--Hausdorff distance $GH(X,Y)$ is as (half, sometimes) the greatest lower bound of the set of $\epsilon>0$ admitting a surjective $\epsilon$-approximation between $X$ and $Y$. It is worth noting that estimation of Gromov--Hausdorff distances in the context of injective spaces and the use of $\epsilon$-approximations are not new to this field, for example: in \cite{Lang_Pavon_Zust-tight_spans} Lang, Pav\'on and Z\"ust estimate the change in Gromov--Hausdorff distance between spaces as one passes from the spaces to their injective envelopes. Thus, we expect that neither our limit lemma nor its proof below come as a surprise to the experts. We include the proof for convenience. 

\medskip
Using $\epsilon$-approximations, Gromov--Hausdorff convergence is characterized as follows.
\begin{lemma}\label{lemma:GH convergence} Let $X_n$, $n\in\NN$ and $X$ be metric spaces. The sequence $(X_n)_{n=1}^\infty$ converges to $X$ in the Gromov--Hausdorff topology, if for every $\epsilon>0$ there exist $N\in\NN$ such that $\epsilon$-approximations $R_n(\epsilon)\subset X_n\times X$ exist for all $n\geq N$.\ep
\end{lemma}
This kind of convergence is most meaningful in the category of compact metric spaces. Pointed Gromov--Hausdorff convergence (of proper metric spaces) extends this idea and is defined as follows:
\begin{definition}[\cite{Bridson_Swarup-GH_convergence}, Definition 1.8] A sequence $(X_n,b_n)_{n\in\NN}$ of pointed proper metric spaces is said to converge to the pointed space $(X,b)$ if for every $R>0$ the sequence of closed metric balls $B_{X_n}(b_n,R)$, $n\in\NN$, converges to the closed ball $B_X(x,R)$ in the Gromov--Hausdorff topology. 
\end{definition}
Note that the completion of a pointed Gromov--Hausdorff limit of a sequence of spaces is itself a limit of the same sequence.

\medskip
We are ready to prove the limit lemma from the introduction.
\begin{lemma}\label{proof of limit lemma} Let $(X_n,b_n)$ be a sequence of pointed proper injective metric spaces converging to the space $(X,b)$ in the sense of pointed Gromov--Hausdorff convergence. Then the metric completion of $X$ is injective. 
\end{lemma}
\begin{proof} We may assume $X$ is complete. We first make an independent observation: let $X,Y$ be metric spaces with $X$ injective and suppose that $R\subset X\times Y$ is an $\epsilon$-approximation. Given a collection $\{y_i\}_{i=1}^k$ of points in $Y$ and positive numbers $\{r_i\}_{i=1}^k$ satisfying
\begin{equation}\label{eqn:hypermetric}
	r_i+r_j\geq\dist{y_i}{y_j}
\end{equation}
for all $i,j\in\{1,\ldots,k\}$. We find points $x_i\in X$, $i=1,\ldots,k$, such that $(x_i,y_i)\in R$ and we set $R_i=r_i+\epsilon$. It then follows that 
\begin{equation}
	R_i+R_j\geq\dist{x_i}{x_j}
\end{equation}
for all $i,j$. Applying Theorem \ref{injectivity criterion} we find a point $x\in X$ with $\dist{x}{x_i}\leq R_i$ for all $i$, and conclude the existence of a point $y\in Y$ which then must satisfy $\dist{y}{y_i}\leq r_i+2\epsilon$ for all $i$.

\medskip
Now we apply the preceding observation to a fixed sequence of pointed injective spaces $(X_n,b_n)$ converging to a space $(X,b)$. Given finite collections $\{x_i\}_{i=1}^k$ of points in $X$ and $\{r_i\}_{i=1}^k$ of positive reals satisfying \eqref{eqn:hypermetric}, find $R>0$ large enough so that $B_X(b,R)$ contains the union of all the balls $B_X(x_i,r_i+2)$, $i=1,\ldots,k$. By the preceding argument, for every $\epsilon\in(0,1)$, the closed ball $B_X(b,R)$ contains a point $x_\epsilon$ satisfying $\dist{x_\epsilon}{x_i}\leq r_i+2\epsilon$ for $i=1,\ldots,k$. The existence of a point $x\in\bigcap B_X(x_i,r_i)$ now follows from the compactness of $B_X(b,R)$.
\end{proof}
\begin{remark}\label{reduction to integer radii} Observe that a different choice of the $R_i$---any choice, in fact, of $r_i+\epsilon\leq R_i\leq r_i+K\epsilon$ for a value of $K$ that is fixed in advance---would do. This enables a weakening of the assumptions (on the $X_n$) under which the preceding argument remains intact. Instead of assuming injectivity of each $X_n$ one may assume that the Aronszajn-Panitchpakdi condition holds for: 
\begin{itemize}
	\item[-] Finite subsets of a fixed $\epsilon$-dense subset $A_n$ of $X_n$,
	\item[-] Radii restricted to, say, values in the set $\tfrac{1}{n}\NN$. 
\end{itemize}
The injectivity of $X$ would follow from these assumptions by the same argument as before.
\end{remark}

\subsection{Median Spaces}\label{section:median spaces} The hyper-convexity criterion of Aronszajn and Panitchpakdi cannot help but remind one of a similar phenomenon in the realm of median spaces. In this short review we follow the exposition of \cite{Chatterji_Drutu_Haglund-measured_spaces_with_walls}. Let $(X,d)$ be a pseudo-metric space. 
\begin{definition}[Intervals and Medians] For $x,y\in X$, the {\it interval} $I(x,y)$ is defined to be
\begin{displaymath}
	I(x,y)=\set{z\in X}{d(x,z)+d(z,y)=d(x,y)}.
\end{displaymath}

For $x,y,z\in X$, we say that a point $m\in X$ is a {\it median} of the triple $(x,y,z)$ if
\begin{displaymath}
	m\in M(x,y,z):=I(x,y)\cap I(x,z)\cap I(y,z).
\end{displaymath}
\end{definition}
By definition, $M(x,y,z)$ is independent of the ordering of $x,y,z$. We also recall the following definitions. 
\begin{definition}[Convexity and Half-spaces] A subset $K$ in a pseudo-metric space $(X,d)$ is {\it convex}, if $I(x,y)\subseteq K$ holds for every $x,y\in K$. $K$ is said to be a {\it half-space of $X$} if both $K$ and $X\minus K$ are convex subspaces of $X$.
\end{definition}
The notion of median points is significant for a large class of spaces.
\begin{definition}[Median Space] A pseudo-metric space is a {\it median space}, if $M(x,y,z)\neq\varnothing$ and $\diam{M(x,y,z)}=0$ for all $x,y,z\in X$. A map $f\brr X\to Y$ between median spaces is said to be a {\it median morphism} if $f(M(x,y,z))\subseteq M(f(x),f(y),f(z))$ for all $x,y,z\in X$.
\end{definition}
Recall that any pseudo-metric can be {\it Hausdorffified}, i.e. made into a metric: a pseudo-metric space $X$ gives rise to a metric space $\hau{X}$ by forming the quotient of $X$ by the equivalence relation $x\sim y\IFF d(x,y)=0$. It is clear then that the median structure on $X$ descends to the quotient.  
\begin{lemma} Suppose $(X,d)$ is a median pseudo-metric space. Then $\hau{X}$ is a median metric space where $M(x,y,z)$ is a singleton for all $x,y,z\in X$. When this happens, we write
\begin{displaymath}
	M(x,y,z)=\left\{\med_X(x,y,z)\right\}
\end{displaymath}
and say that $\med_X(x,y,z)$ is the {\it median point} of the triple $(x,y,z)$.
\end{lemma}
We provide a few simple examples of median spaces to illustrate this structure. 
\begin{example}[The real line is median] It is easy to verify that $\RR$ is a complete  geodesic median space, with
\begin{equation}\label{real median}
	\med_\RR(x,y,z)=y\text{ whenever }x\leq y\leq z
\end{equation}  
\end{example}

\begin{example}[$\ell_1$-type normed spaces are geodesic median spaces]\label{std median ops} Suppose $(T,\mu)$ is a measure space. Then $X=L^1(T,\mu)$ is a median space and
\begin{equation}
	\med(x,y,z)=\left(\med_\RR(x(t),y(t),z(t))\right)_{t\in T}
\end{equation}
In the case when $T$ is finite, we see that $X$ is simply isometric to $\ell_1(T):=(\RR^T,\Vert\cdot\Vert_1)$, with medians computed coordinatewise. Also observe that the subset $\{0,1\}^T$ forms a median subspace, in which the median operation reduces to a (pointwise) majority vote.
\end{example}

\medskip
Two things make median spaces highly relevant to this paper: the first is the fact that every cubing can be metrized to become a median space (and in more than one way as we shall see in section \ref{section:weighted cubings}); the second is the following theorem we have already mentioned in the introduction.
\begin{theorem}[Helly Theorem, \cite{Roller-duality} Theorem $2.2$]\label{Helly theorem} Suppose $\hsm{C}=\{C_i\}_{i=1}^n$ is a family of convex subsets of a median space $X$. If any two elements of $\hsm{C}$ have a point in common, then all elements of $\hsm{C}$ have a point in common.
\end{theorem}
Half-spaces play a crucial role in the theory of median spaces, as described thoroughly in~\cite{Chatterji_Drutu_Haglund-measured_spaces_with_walls}. For example, given a convex subset $Y$ of a median space, the subset $Y$ can be recovered by considering half-spaces separating $Y$ from singletons in the complement (see Theorem 4.8 of \cite{Chatterji_Drutu_Haglund-measured_spaces_with_walls}). If the median space is complete and $Y$ is closed, we can also find sufficiently many separating closed half-spaces since points in the complement have positive distance from $Y$. We summarize this as follows. 
\begin{proposition}\label{halfspace presentation of convex sets}\label{convex in median is V} Let $X$ be a complete median space. Then every convex subset of $X$ is an intersection of half-spaces. Any closed convex subset of $X$ is an intersection of {\it closed} half-spaces.
\end{proposition}

As explained in the introduction, our strategy for verifying the injectivity of a piecewise-$\ell_\infty$ cubing $X$ is to show that closed balls in such a cubing are {\it convex with respect to the piecewise-$\ell_1$ metric on the same cubing}. By Helly's theorem it will suffice to demonstrate that any closed $\ell_\infty$-ball is the intersection of closed half-spaces of $X$ when $X$ is viewed as a median space (using its piecewise-$\ell_1$ metric). The convexity of $\ell_\infty$-balls will result from us presenting any such ball as the intersection of a suitably chosen family of (closed) $\ell_1$-halfspaces.

\bigskip
The halfspace structure of a median space is fundamental to our technique. In fact, we will make good use of an isometric median embedding $\rho_{_X}$ of a median space $X$ into an $L^1$ space, constructed as follows (Theorem 5.1 and Corollary 5.3 in \cite{Chatterji_Drutu_Haglund-measured_spaces_with_walls}):
\begin{itemize}
	\item{\bf Construction of a `Transverse Measure', $\mu$: } one starts by constructing the set $\hsm{W}(X)$ of all pairs of the form $\{h,X\minus h\}$ where $h$ ranges over the halfspaces of $X$; it then turns out that a natural $\sigma$-algebra $\hsm{B}(X)$ exists on $\hsm{W}(X)$, together with a measure $\mu$, such that $\mu\hsm{W}(x|y)=\dist{x}{y}$ for all $x,y\in X$, where $\hsm{W}(x|y)$ denotes the subset of all pairs $\{h,X\minus h\}\in\hsm{W}(X)$ satisfying $x\in h$ and $y\in X\minus h$.
	\item{\bf Embedding in $L^1(\mu)$: } fixing a base-point $b\in X$, map a point $x\in X$ to the function $\charf{\hsm{W}(x|b)}$. This map is a median-preserving isometric embedding.
\end{itemize}

\medskip
The case of locally-finite (or, equivalently, proper) cubings is well known (\cite{Sageev-thesis,Roller-duality,Chepoi-median_graphs,Niblo_Roller-property_T}) to fall under the purview of this construction while lending itself to analysis by simple combinatorial (rather than measure-theoretic) tools. We seek to capitalize on this fact in Section~\ref{section:geometry of cubings}, where we use this embedding, realized as a weighted counting measure (see, e.g. Lemma~\ref{lemma:rep map}), to study the geometry of cubings endowed with a piecewise-$\ell_\infty$ geometry.

\medskip
Finally, some quick observations regarding the ambiguous relationship between median spaces and injective spaces are in order. Analyzing Isbell's construction of injective envelopes it is easy to verify that, if $X$ is an injective space and $Y$ is a subspace of $X$ then the inclusion of $Y$ into $X$ extends to an embedding of the envelope $\epsilon Y$ in $X$. The same construction enables one to prove that the injective envelope of three points is either an interval (degenerate or non-degenerate) or a tripod---a space isometric to a metric tree with four vertices and three leaves standing in bijective correspondence with the initial triple of points. Thus, for any triple of points $x,y,z$ in an injective space $X$, the median set $M(x,y,z)$ is non-empty!

\medskip
While medians do exist in injective spaces, there is no guarantee of uniqueness of the median, as one can easily observe in the normed space $\ell_\infty(S)$ for any set $S$ of cardinality greater than $2$. In fact, even cases so simple as the case of four points demonstrate the complicated relationship between the two classes---consider the following example.
\begin{figure}[t]
	\begin{center}
		\includegraphics[width=.85\columnwidth]{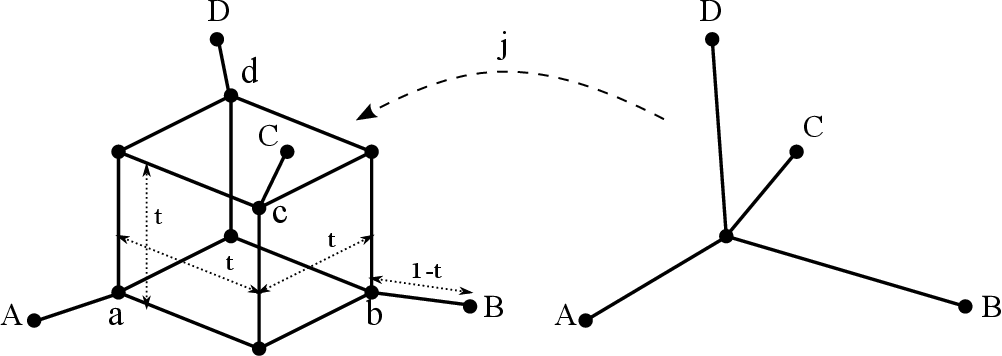}
	\end{center}
	\caption{The spaces $X_t$ (left) and $X_0$ (right) from Example~\ref{example:non-injective 1}. A median-preserving map from $X_0$ to $X_t$ is impossible.\label{fig:non-injective median space}}
\end{figure}
\begin{example}[Non-Injective Median Spaces]\label{example:non-injective 1} Let $t\in[0,1]$ be a real parameter, and let $Q_t$ denote the three-dimensional cube $[0,t]^3\subset\RR^3$, endowed with the $\ell_1$ metric. Denote:
\begin{displaymath}
	a=(0,0,0)\,,\;\;
	b=(t,t,0)\,,\;\;
	c=(t,0,t)\,,\;\;
	d=(0,t,t)
\end{displaymath}
and let $J_x$, $x\in\{a,b,c,d\}$ denote pairwise disjoint copies of the the interval $[0,1-t]\subset\RR$ endowed with the standard metric. We form a space $X_t$ as the quotient of $Q_t\cup\{J_a,J_b,J_c,J_d\}$ by identifying the point $0\in J_\sigma$ with the point $\sigma\in Q_t$ for every $\sigma\in\{a,b,c,d\}$, and endow $X_t$ with the quotient metric. Finally, we denote the point in $X_t$ corresponding to $1-t\in J_\sigma$ with the capital letter variant of $\sigma\in\{a,b,c,d\}$; see Figure~\ref{fig:non-injective median space}.

Deferring the proof that the $X_t$ are median spaces until Section~\ref{section:weighted cubings}, Example~\ref{example:non-injective 2}, let us focus on explaining why $X_t$ is not injective for any $t>0$. Observe that the subspace $Y=\{A,B,C,D\}\subset X_t$ satisfies $\dist{x}{y}=2$ for all distinct $x,y\in Y$. It follows that $X_0$ is the injective envelope of $Y$ with this metric (see example and discussion at the end of Section~\ref{injective envelopes}). If $X_t$ were injective for some $t>0$, the inclusion map $inc:Y\to X_t$ would have extended to an isometric embedding $j:X_0\to X_t$. Since both $X_0$ and $X_t$ are median spaces, $j$ is then a median-preserving map. Denoting the median map of $X_t$ by $\med_t$, we observe in $X_t$ that
\begin{eqnarray*}
	j\left(\med_0(A,B,C)\right)&=&\med_t(A,B,C)=c\,,\\
	j\left(\med_0(A,B,D)\right)&=&\med_t(A,B,D)=d,
\end{eqnarray*}
while at the same time in $X_0$ one has
\begin{displaymath}
	\med_0(A,B,C)=c=d=\med_0(A,B,D)\text{ in }X_0.
\end{displaymath}
Since $c\neq d$ in $X_t$ for $t>0$, we have arrived at a contradiction and we conclude $X_t$ cannot be injective for $t>0$.\ep
\end{example}
The last example hints that failure of injectivity is fairly common among median spaces. We find it intriguing, then, that a mere change of the geometry of the vector space on which the individual cubes are modeled should result in a change so radical as turning all the spaces in question into injective spaces. This kind of behaviour seems to hint at the existence of a general principle loosely formulated as {\it a non-positively curved combination of injective spaces is injective}. The extent to which such a statement is true remains to be verified.

\section{Piecewise-$\ell_1$ Cubings and their Geometry}\label{section:geometry of cubings}
For a very serious and inspiring recent account of the theory of non-positively curved cubical complexes please see \cite{Wise-riches_to_raags}. For a detailed account of their formal underpinnings as unit piecewise-$\ell_2$ cubings (of course, these are precisely the CAT(0) cubical complexes briefly reviewed in Section~\ref{NPC}), see  \cite{Bridson_Haefliger-nonpositive_curvature}, Section I.7 and~\cite{Leary}, Appendices A-C.

\medskip
Beyond these references, we will not delve into any detail regarding cubical complexes from the topological viewpoint. Instead we will focus on a very direct construction of cubings, due initially to Sageev \cite{Sageev-thesis}, and developed by Roller in \cite{Roller-duality} based on Isbell's duality \cite{Isbell-duality} between poc sets and median algebras. Let us just clarify what we mean by a cube, to avoid any confusion.
\begin{definition}\label{defn:standard cube} Let $S$ be a set. The {\it standard $S$-cube} is the set $\cub{S}$ of all functions from $S$ into $[0,1]$. When $S=\varnothing$, the {\it 0-dimensional cube} $\cub{S}$ is defined to equal the one-point set $\{\ast\}$. More generally, we say that $\cub{S}$ is a $\card{S}$-dimensional cube.
\end{definition}
Much of the material present in this section has been known to geometric groups theorists for quite some time, either formally or as folklore. Unfortunately, the literature on the technical tools we are using is in a state of slight disarray preventing immediate application to the specific problem at hand. We therefore decided to gather the necessary material here, filling in some of the gaps and formalizing the folklore.

\subsection{Poc Sets}\label{poc sets} We recall some definitions and examples.
\begin{definition}[Poc Set, \cite{Roller-duality}] A {\it poc set} is a poset $(P,\leq)$ with a minimum element denoted $0\in P$ and endowed with an order-reversing involution $a\mapsto a^\ast$ satisfying the additional requirement
\begin{displaymath}
	(\dagger)\quad
	a\leq a^\ast\THEN a=0.
\end{displaymath}
We say that $P$ is discrete, if the poset $(P,\leq)$ is discrete, that is: order intervals
\begin{equation}
	[a,b]=\set{x\in }{a\leq x\leq b}
\end{equation}
are finite for all $a,b\in P$.
\end{definition}
Working with poc sets requires some additional jargon.
\begin{definition} Let $P$ be a poc set. 
\begin{itemize}
	\item[-] Elements of $P$ are often referred to as {\it halfspaces};
	\item[-] The elements $0,0^\ast\in P$ are said to be {\it trivial};
	\item[-] The non-trivial halfspaces are said to be {\it proper};
	\item[-] A complementary pair $\{a,a^\ast\}$ of proper halfspaces is called a {\it wall} of $P$;
	\item[-] A pair $\{a,b\}$ with $a\notin\{b,b^\ast\}$ will be called a {\it proper pair}. 
\end{itemize}
\end{definition}
The meaning of $(\dagger)$ is summarized in the observation that any proper pair $a,b\in P$ satisfies {\it at most one of the following relations}:
\begin{equation}\label{nesting relations}
	a\leq b\,,\;
	a^\ast\leq b\,,\;
	a\leq b^\ast\,,\;
	a^\ast\leq b^\ast.
\end{equation}
The above relations (if they hold) are called {\it nesting relations}.
\begin{definition} Let $P$ be a poc set. A pair of elements $a,b\in P$ satisfying one of the relations \eqref{nesting relations} is said to be {\it nested}. More generally, a set $A\subset P$ is said to be nested if its elements are nested pairwise. The set $A$ is said to be {\it transverse} if no two of its elements are nested. A transverse pair $a,b\in P$ is often denoted with $a\pitchfork b$.
\end{definition}
\begin{example}[Power Set] The power set $\power{S}$ of a non-empty set $S$ is a poc set with respect to inclusion and complementation. In particular, the power set of one point, denoted $\power{}$, is the {\it trivial} poc set.
\end{example}
\begin{example}[Linear Poc Set]\label{example:linear poc set} Let $(P,\leq)$ be a totally ordered set. Then the set $Q=P\sqcup P^\ast\sqcup\{\minP,\maxP\}$ -- where $P^\ast$ is the set of symbols of the form $p^\ast$ for $p\in P$ -- and subject to the relations $\minP\leq x$ and $x\leq\maxP$ for all $x\in Q$, as well as the relations $p^\ast\leq q^\ast$ iff $q<p$ in $P$, is a poc set.
\end{example}
\begin{example}[Poc Set of a Tree]\label{example:poc set of a tree} Let $T$ be a tree with edge set $E$ and vertex set $V$. Let $P$ be the set of vertex-sets of connected components of all possible $T-e$, where $e$ ranges over $E$. Then $P$ is a discrete nested poc set with respect to inclusion and complementation.
\end{example}
Of course, one should also mention the morphisms in this category:
\begin{definition} A function $f:P\to Q$ between poc sets is a {\it poc morphism}, if
\begin{equation}
		a\leq b\THEN f(a)\leq f(b)\,,\quad f(a^\ast)=f(a)^\ast
\end{equation}
hold for all $a,b\in P$, and $f(\minP)=\minP$. The set of all poc morphisms from $P$ to $Q$ is denoted by $\mathrm{Hom}(P,Q)$.
\end{definition} 

\subsection{The Dual Cubing of a Poc Set}
The following construction is due to Sageev \cite{Sageev-thesis} in the specific context of relative ends of groups, and to Isbell \cite{Isbell-duality} and Roller \cite{Roller-duality} in the current level of generality.
\begin{definition}[Dual of a Poc Set]\label{dual} Let $P$ be a finite poc set. An {\it ultra-filter} $U$ on $P$ is a subset of $P$ satisfying:
\begin{enumerate}
	\item\label{ast-selection} for all $a\in P$ one has $a\in U$ or $a^\ast\in U$, but not both;
	\item\label{coherence} for all $a,b\in U$ one {\it never} has $a\leq b^\ast$.
\end{enumerate}
The set of all ultra-filters on $P$ is denoted by $P^\circ$. A subset $U\subset P$ satisfying \eqref{ast-selection} is called a {\it complete $\ast$-selection on $P$}. A subset $U\subset P$ satisfying \eqref{coherence} is said to be {\it coherent}. We will denote the set of all $\ast$-selections by $S(P)^0$, in anticipation of the full complex of $\ast$-selections, which will be defined below, and denoted by $S(P)$.
\end{definition}
\begin{remark} It is convenient to denote $A^\ast=\{a^\ast\,|\,a\in A\}$ for subsets $A$ of a poc set $P$. Thus, $U\subseteq P$ is a complete $\ast$-selection if and only if $U\cap U^\ast=\varnothing$ and $U\cup U^\ast=P$. A subset $U$ satisfying the first requirement (but possibly not the second) is called a {\em $\ast$-selection} on $P$.
\end{remark}
\begin{remark} Note that $\minP\notin U$ for all $U\in P^\circ$ because of coherence, and hence also $\minP^\ast\in U$, because $U$ is a complete $\ast$-selection.
\end{remark}
\begin{remark}\label{remark:dual maps} Also note that the mapping $f\mapsto f\inv(1)$ provides a natural identification of $\mathrm{Hom}(P,\power{})$ with $P^\circ$. An easy consequence is that any $f\in\mathrm{Hom}(P,Q)$ induces a mapping $f^\circ:Q^\circ\to P^\circ$, $U\mapsto f\inv(U)$, called {\em the dual of $f$}, which, under the above identification takes the form $\varphi\mapsto\varphi\circ f$, $\varphi\in\mathrm{Hom}(Q,\power{})$. Since $\varphi\circ f$ is a poc morphism (being the composition of two poc morphisms), the mapping $f^\circ$ is well-defined in the sense that $f\inv(U)\in P^\circ$ holds whenever $U\in Q^\circ$.
\end{remark}

\medskip
Picking a basepoint $B\in S(P)^0$ makes it easy to identify $S(P)^0$ with the vertex set of $\cub{B}$ (recall Definition~\ref{defn:standard cube}). The identification map $\rho:S(P)^0\to\cub{B}$ is defined by 
\begin{equation}\label{eqn:representation map}
	\rho:U\mapsto\charf{B\minus U}\,,
\end{equation}
and we need to verify some of its properties:
\begin{lemma}[Metric on $S(P)^0$]\label{lemma:rep map} The map $\rho$ of Equation \eqref{eqn:representation map} maps $S(P)^0$ bijectively onto $\power{B}$, the set of $\{0,1\}$-valued functions on $B$. Furthermore, the pullback metric on $S(P)^0$ defined by $\rho$,
\begin{displaymath}
	\ellone{}(U,V):=\Vert \rho(U)-\rho(V)\Vert_1\,,
\end{displaymath}
satisfies the identities:
\begin{equation}\label{ellone metric on a cubing}
	\ellone{}(U,V)=\card{U\minus V}=\card{V\minus U}=\tfrac{1}{2}\card{U\symplus V}
\end{equation}
(Compare with the map defined at the end of Section~\ref{section:median spaces}).
\end{lemma}
\begin{proof} Surjection is obvious. To prove injectivity, first observe that, for any $U,V\in S(P)^0$, one has $U\minus V=U\cap V^\ast$. In particular, given $B,U$ as above and $\eta=\rho(U)$, one first recovers $B\cap U^\ast$ as $\eta\inv(1)$, then $U\cap B^\ast$ as $(B\cap U^\ast)^\ast$ and $U\cap B$ as $B\minus(B\cap U^\ast)$. Finally, $U$ is recovered as $(B\cap U)\cup(B\cap U^\ast)$.

To prove the identities claimed for $\rho$, first observe that 
\begin{displaymath}
	U\minus V=U\cap V^\ast=(U^\ast\cap V)^\ast=(V\minus U)^\ast\,,
\end{displaymath}
which proves all but the first equality in \eqref{ellone metric on a cubing}. 
\begin{displaymath}
	\ellone{}(U,V)=\sum_{b\in B}|\charf{B\minus U}(b)-\charf{B\minus V}(b)|=\sum_{b\in B}|\charf{B\cap U^\ast}(b)-\charf{B\cap V^\ast}|\,.
\end{displaymath}
Each summand is $0$ or $1$, with a contribution of $1$ being made if and only if $b\in U^\ast\minus V^\ast$ or $b\in V^\ast\minus U^\ast$ (note that this is an exclusive `or'); equivalently, $b\in U^\ast\cap V$ or $b\in V^\ast\cap U$; equivalently, $b\in V\minus U$ or $b\in U\minus V$; equivalently, $b\in V\minus U$ or $b\in(V\minus U)^\ast$; equivalently, $b\in V\minus U$ or $b^\ast\in V\minus U$. Thus, the non-zero summands above are in one-to-one correspondence to the elements of $V\minus U$, finishing the proof.
\end{proof}

An edge of the cube $\cub{B}$ is a pair of $\{0,1\}$-valued functions on $B$ which differ in one point only. Pulling the edges of $\cub{B}$ back along $\rho$ endows $S(P)^0$ with the structure of a graph, denoted $S(P)^1$, where two vertices $U,V\in S(P)^0$ are joined by an edge if and only if $\ellone{}(U,V)=1$.
 
More generally, for any $A\subseteq B$ and any function $f:A\to\power{}$, the set of functions $\tilde f\in\cub{B}$ with $\tilde f\big|_A=f$ form a face of $\cub{B}$. The preimage of this set under $\rho$ coincides with the set of all complete $\ast$-selections $U\in S(P)^0$ containing the partial selection $(A\cap f\inv(1))\cup(A^\ast\cap f\inv(0))$, and may be thought of as the set of vertices in $S(P)^0$ spanning a face in a cubical complex, denoted $S(P)$, which is then isomrphic to $\cub{B}$. However, since $S(P)$ is completely determined by the metric $\ellone{}$, which is independent of the choice of $B$, so is $S(P)$ itself independent of the choice of $B$.


\begin{definition} Let $P$ be a discrete poc set. Then $\cube{}{P}$ is defined to be the sub-complex of $S(P)$ of all faces {\em not} incident on an incoherent vertex. Equivalently, it is the cubed sub-complex of $S(P)$ induced by $P^\circ$.
\end{definition}
A few simple examples are as follows. 
\begin{example}[Dual of an Orthogonal Poc Set] When the poc set $P$ has no non-trivial nesting relations it is clear that $\cube{}{P}$ coincides with the cube $S(P)$.
\end{example}
\begin{example}[Dual of a Linear Poc Set]\label{dual of linear} It is not hard to see that $\cube{}{P}$ for a linear poc set $P$ (see Example~\ref{example:linear poc set}) is the extension of $P$ by Dedekind cuts. 
\end{example}
\begin{example}[Dual of a Nested Poc Set]\label{dual of nested} It is a more involved computation to verify that the dual of a finite nested poc set $P$ is a finite tree\footnote{In fact, this case is also a particular case of Dunwoody's tree construction (see~\cite{Dicks_Dunwoody-groups_acting_on_graphs}, Theorem II.1.5), forming a tree from a nested system $E$ of subsets of a set $V$, that is: every $e,f\in E$ satisfy one of $e\subseteq f$, $e^\ast\subseteq f$, $e^\ast\subseteq f^\ast$ or $e\subseteq f^\ast$, where $e^\ast:=U\minus e$. If $E$ is contained in an almost-equality class of $\power{U}$ (that is, $e\symplus f$ is finite for all $e,f\in E$), one could think of each $e\in E$ as a directed edge of a tree $T$, where the support of $e$ itself is seen as the set of ``generalized vertices'' lying in the connected component of the head of $e$ in the two-component forest $T-e$.}. Moreover, $P$ is naturally isomorphic to the poc set constructed from this tree as described in example \ref{example:poc set of a tree}.
\end{example}

Finally, we give an example we will use later on.
\begin{example}[Cartesian Products]\label{Cartesian products} Suppose $P$ and $Q$ are discrete poc sets, and let $P\vee Q$ denote the poc set obtained from $P\cup Q$ by identifying $\minP_P$ with $\minP_Q$ (and hence also $\maxP_P$ with $\maxP_Q$). Then any proper element of $P$ is transverse to any proper element of $Q$ and it is easily shown that $\cube{}{P\vee Q}$ is naturally isomorphic to $\cube{}{P}\times\cube{}{Q}$ by employing the duals of the inclusion morphisms $P\to P\vee Q$ and $Q\to P\vee Q$ (see Remark~\ref{remark:dual maps}).
\end{example}

Summarizing all the above is the surprising result anticipated by Sageev and proved by Roller and Chepoi\footnote{In his thesis~\cite{Sageev-thesis}, aiming to study ends of group pairs ($G$ a finitely generated group, $H$ a finitely generated subgroup), Sageev constructed a cubing from a complement-closed collection of $G$-translates of certain right $H$-invariant subsets $G$ by forming a dual with respect to the containment order. Roller realized in~\cite{Roller-duality} the generality of this construction and its special place in the duality between median algebras and poc sets discovered by Isbell~\cite{Isbell-duality}. He studied this duality in depth, reformulating Sageev's work in terms of discrete poc sets, constructing the dual, $P^\circ$, of a discrete poc set $P$ as a Stone median algebra, each almost-equality class of which is a discrete median algebra (see {\it loc.~cit.}, Theorems 5.3, 6.4). Among other topologically flavored questions, Roller studied the question of when a discrete poc set $P$ may be naturally associated with a connected component $C$ of $\cube{}{P}$ whose standard set of half-spaces recovers $P$ ({\it loc.~cit.}, Section 9). This gave rise to the notion of the {\em Roller boundary} of an infinite cubing (first mentioned under this name in~\cite{Nevo_Sageev-poisson_bdry_for_cubings}), which is the residual set of the closure of $C$ in $\cube{}{P}$ with respect to the Tychonoff topology. The particular formulation in this paper, realizing $\cube{}{P}$ as a sub-complex of $S(P)$ produced by `excavating' the incoherent $\ast$-selections, offers a precise alternative to the construction one generally finds in the Geometric Group Theory literature, where one usually rigorously constructs the $1$-skeleton of $\cube{}{P}$ and then ``glues in cubes inductively so as to satisfy the flag condition''. Here, instead, we prefer an exposition closer to that of Chepoi (see \cite{Chepoi-median_graphs}, Theorem 6.1), who proves directly that the complexes arising as connected components of $\cube{}{P}$ (see Section~\ref{cubings are median}), are cubings. Chepoi's proof that such complexes are simply connected strengthens Sageev's ``disk-diagram'' argument, extending it over additional classes of complexes and endowing them with CAT(0) geometries~ ---~ see {\em loc.~cit.}, Section 7).}:
\begin{theorem}[Sageev--Roller, Chepoi]\label{Sageev-Roller duality} Let $P$ be a discrete poc set. Then every connected component of $\cube{}{P}$ is a cubing and $\ellone{}$ coincides with the combinatorial path metric on its $1$-dimensional skeleton. Conversely, every cubing $X$ is a connected component of $\cube{}{P}$ for an appropriately chosen discrete poc set $P$.
\end{theorem}
When $P$ is infinite, $\cube{}{P}$ inevitably forms a disconnected space. The connected components of $\cube{}{P}$ are spanned by the almost-equality classes of its vertices: recall that two subsets $U,W$ of a set $S$ are said to be {\it almost equal} if the symmetric difference $U\symplus W$ is finite; it follows from the theorem and from Lemma~\ref{lemma:rep map} that a pair of vertices in $\cube{}{P}$ is joined by a (finite) edge-path if and only if they are almost-equal to each other as subsets of $P$. For example, if $P$ is discrete and nested, $\cube{}{P}$ will consist of a tree whose poc set of halfspaces is naturally isomorphic to $P$, together with its space of ends, each end being the only point of its component in $\cube{}{P}$ (compare with~\cite{Dicks_Dunwoody-groups_acting_on_graphs}, Corollary II.1.10 for this characterization of trees and Paragraph IV.6.3 for the discussion of their ends). 

\medskip
In general, however, it may not be possible to {\it naturally} select a (distinguished) component $K$ of $\cube{}{P}$ whose poc set of half-spaces is isomorphic to $P$. For example, if $P$ is countably infinite and {\em orthogonal}~ ---~ that is, if no proper pair in $P$ is nested~ ---~ then $\cube{}{P}$ coincides with $S(P)$, and there is no preferred vertex. Section 9 in \cite{Roller-duality} and Section 3.1 in \cite{Guralnik-Roller_boundary} explains how to achieve this under the condition that $P$ contains no infinite transverse set.

\medskip
Seeking to avoid these technicalities we will use the trick of restricting attention to a component containing a particular vertex of interest.
\begin{definition}\label{picking a component} Let $P$ be a discrete poc set and $B\in P^\circ$. The connected component of the complex $\cube{}{P}$ containing the vertex $B$ will be denoted $\cube{}{P}_B$.
\end{definition}
Recall the (piecewise affinely extended) representation map $\rho \colon \cube{}{P}\to\RR^B$ defined in Equation~\ref{eqn:representation map}, and note that $\cube{}{P}_B$ is precisely the pre-image under $\rho$ of the vector subspace $\ell_1(B)$ of $\RR^B$ consisting of all vectors of finite $1$-norm. This is due to $P$ being discrete.

\subsection{Local Properties of Duals}\label{section:local props of duals} We will need some technical information about the local structure of $\cube{}{P}$. The following results are well-known and appear in \cite{Sageev-thesis,Roller-duality}, and provide the backbone for all the available variants of Theorem~\ref{Sageev-Roller duality}.
\begin{lemma}[vertex links in $\cube{}{P}$]\label{lemma:links and minimality} Let $P$ be a discrete poc set and let $V,V'\in P^\circ$ be vertices of $\cube{}{P}$. Then $V'$ is adjacent to $V$ in $\cube{}{P}$ if and only if $V'$ has the form
\begin{equation}\label{eqn:neighbour}
	V'=\flip{a}{V}:=(V-\{a\})\cup\{a^\ast\}
\end{equation}
for $a\in\min(V)$, where $\min(V)$ denotes the set of minimal elements of $V$ with respect to the ordering induced on it from $P$.
\end{lemma}
The operation of replacing $V$ by $\flip{a}{V}$ will be called a {\it flip}. Clearly, any vertex of $\cube{}{P}_B$ (for any fixed choice of $B\in P^\circ$) is connected to any other by a finite sequence of flips. The fact that $\ellone{}(U,V)$ simply measures the minimal number of flips required to to turn $U$ into $V$ is central to the theory of discrete median algebras.

\medskip
A special situation occurs when several flips may be applied to a vertex in different orders of application without affecting the outcome. It is easy to verify that the hyperplanes being flipped form a transverse set in such a case, and, more generally we have the following lemma.
\begin{lemma}[cubes in $\cube{}{P}$]\label{lemma:cubes and transversality} Let $P$ be a discrete poc set and let $V\in P^\circ$ be a vertex in $\cube{}{P}$. Then the set of cubes of $\cube{}{P}$ incident to $V$ is in one-to-one correspondence with transverse sets $A\subset\min(V)$. Moreover, for each such $A$, the vertices of the corresponding cube adjacent to $V$ are all of the form
\begin{equation}\label{eqn:vertices of a cube}
	\flip{a_1,\ldots,a_m}{V}:=\flip{a_m}{\cdots\flip{a_2}{\flip{a_1}{V}}\cdots}
\end{equation}
where $(a_1,\ldots,a_m)$ is a sequence of distinct elements of $A$ and we agree to denote $\flip{\varnothing}{V}:=V$. The vertex $\flip{a_1,\ldots,a_m}{V}$ is independent of the ordering of the flips. 
\end{lemma}
This lemma is very well illustrated by Example~\ref{Cartesian products}, demonstrating how the dimensions of cubes in the complex grow with the sizes of transverse subsets.

\subsection{Cubings as Median Spaces}\label{cubings are median} 
The following fact is a direct application of Example~\ref{std median ops} to our discussion of the representation mapping $\rho$:
\begin{proposition}[Poc set duals are median]\label{poc set duals are median} Let $P$ be a discrete poc set and let $B\in P^\circ$ be a basepoint in $\cube{}{P}$. And let $\rho$ be the representation map defined in Equation~\eqref{eqn:representation map}. Then, each component of $(P^\circ,\ellone{})$ is a median metric space with the median map given by
\begin{equation}
	\med(U,V,W)=(U\cap V)\cup(V\cap W)\cup(W\cap U).
\end{equation}
Moreover, the map $\rho:\cube{}{P}_B\to\ell_1(B)$ is a median-preserving isometric embedding.
\end{proposition}
Once again, we cannot stress enough the importance of $\ellone{}$ being a geodesic metric (in the discrete sense) on every $\cube{}{P}_B$. Underlying this fact are the following conclusions from the explicit form of the median map, arising as a corollary of Theorem 5.3 of~\cite{Roller-duality}, applied to the case of discrete poc sets:
\begin{proposition}[convex subsets of $P^\circ$] Let $P$ be a discrete poc set. The half-spaces of $(P^\circ,\ellone{})$ are precisely the sets of the form
\begin{equation}
	V(a):=\set{U\in P^\circ}{a\in U}\,,\quad a\in P\,.
\end{equation}
In particular, by Proposition~\ref{convex in median is V}, the convex subsets of $P^\circ$ are precisely
\begin{equation}\label{convex subset of Pcirc}
	V(A):=\bigcap_{a\in A}V(a)=\set{U\in P^\circ}{A\subseteq U}
\end{equation}
for $A$ ranging over all coherent subsets of $P$. Moreover,
\begin{equation}
	a<b\text{ in }P\IFF V(a)\subsetneq V(b)
\end{equation}
holds for all $a,b\in P$. Thus $P$ is naturally isomorphic to the poc set of halfspaces of $P^\circ$.
\end{proposition}
This proposition enables the computation of the convex hull of a set of vertices in $P^\circ$, defined, as usual, as the intersection of all convex subsets of $P^\circ$ containing the given set.
\begin{lemma} Let $P$ be a discrete poc set. Then the convex hull of a set $\mathscr{F}$ of vertices in $P^\circ$ equals $V(\bigcap\mathscr{F})$.
\end{lemma}
\begin{proof} If $C$ denotes the convex hull of the collection $\mathscr{F}$ and $F=\bigcap\mathscr{F}$, then, $C$ is clearly contained in $D=V(F)$, since each $U\in\mathscr{F}$ contains $F$, which is the same as to say that $U\in V(F)$. Now, since $C$ is itself an intersection of sets of the form $V(a)$, $a\in P$, if $W\in D\minus C$, then there is $h\in P$ with $W\in V(h)$ and $C\subset V(h^\ast)$. The latter implies $h^\ast\in U$ for every $U\in\mathscr{F}$, which means $h^\ast\in F$. But that could not be, since $W\in V(F)$ and already $h\in W$, which is a contradiction.
\end{proof}
Some straightforward and useful corollaries are:
\begin{corollary}\label{interval} Let $P$ be a discrete poc set. For any $U_1,U_2\in P^\circ$, the interval $I(U_1,U_2)$ coincides with the set of all $W\in P^\circ$ containing $U_1\cap U_2$.
\end{corollary}
\begin{proof} Since intervals in median spaces are convex, $I(U_1,U_2)$ is the convex hull of $\{U_1,U_2\}$. Now, set $\mathscr{F}=\{U_1,U_2\}$ and apply the last lemma. 
\end{proof}
\begin{corollary}[implies \cite{Leary}, Theorem B.8]\label{finite convex hull} Let $P$ be a discrete poc set and let $B\in P^\circ$. Then any compact subset of $\cube{}{P}_B$ is contained in a finite convex sub-complex (and hence a sub-cubing) of $\cube{}{P}_B$. 
\end{corollary}
\begin{proof} Let $K\subset\cube{}{P}_B$ be the compact subset in question. Without loss of generality, $K$ is a finite subset of $P^\circ$ (else, replace $K$ with the set of vertices of all cubes which intersect $K$; any convex set containing the new $K$ must contain the original, too).

It therefore suffices to verify that the convex hull of a finite set of vertices in an almost-equality class of $P^\circ$ is finite. We proceed by induction on $|K|$, the base case $|K|=1$ being trivial. It then suffices to prove that if $K\subseteq P^\circ$ is convex and finite and $W\in P^\circ$ lies in the same almost-equality class as $K$, then $K\cup\{W\}$ has a finite convex hull. 

Let $C$ be the convex hull of $K\cup\{W\}$, let $S\subset P$ be the set of all $a\in W$ such that $K\subset V(a^\ast)$ and let $T\subset P$ be the set of all $h\in P$ with $K\in V(h)$. Then $K=V(T)$, $S^\ast\subset T$ and hence, by the last lemma, $C=V(W\cap T)=V(T\minus S)$.

Since $W$ is in the same almost-equality class as $K$, $S$ is finite. Because every ultra-filter in $C$ differs from an ultra-filter in $K$ at most by a complete $\ast$-selection on $S$, we conclude that $|C|\leq |K|\cdot 2^{|S|}$, as desired.
\end{proof}

\bigskip
Our current goal is to extend this combinatorial structure theory of abstract cubings to meet the metric theory of median spaces at a point where we can plainly view finite cubings as `finitely presented median spaces', and proper median spaces as pointed Gromov--Hausdorff limits of finite cubings. For example, one would like to be able to reason in a way hinted at by Figure~\ref{fig:median limit}.

\begin{figure}[t]
	\begin{center}
		\includegraphics[width=\columnwidth]{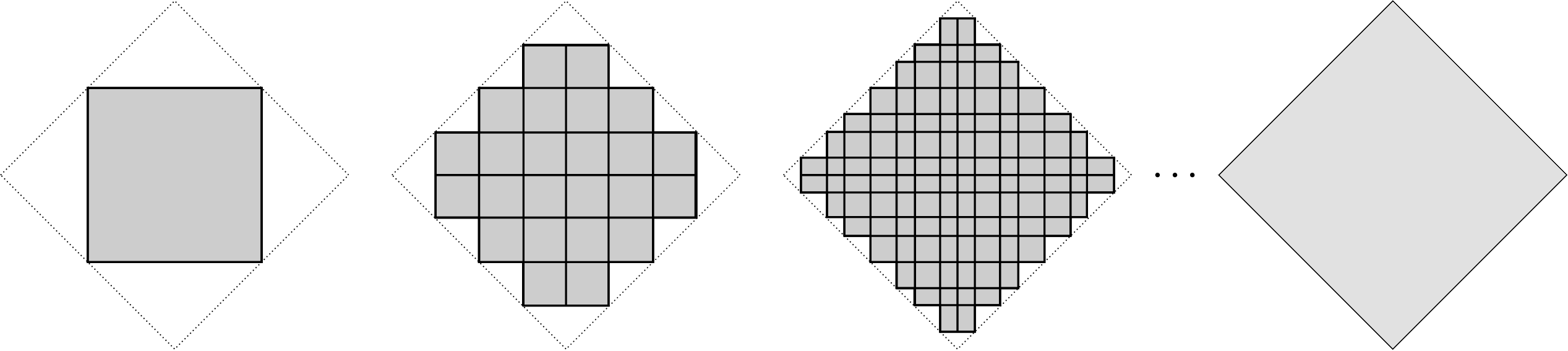}
	\end{center}
	\caption{The unit ball in the $\ell_1$ plane (right) as a limit of a sequence (left to right) of rescaled cubings.\label{fig:median limit}}
\end{figure}

\subsection{Weighted Realizations of $\cube{}{P}$}\label{section:weighted cubings} Let $P$ be a discrete poc set, and let $B$ be a base point in $P^\circ$, fixed once and for all. Recall that $\maxP\in U$ for all $U\in P^\circ$, which makes $\maxP$ an uninformative coordinate of $\RR^B$ when it comes to representing vertices of $\cube{}{P}$.

\medskip
A more varied geometric realization of $\cube{}{P}$ is needed.
\begin{definition}[weight on a poc set] A weight on $P$ is a function $w:P\to\RRplus$ satisfying $w(0)=0$ and $w(a^\ast)=w(a)$ for all $a\in P$. We say that a weight is non-degenerate if $w(a)>0$ for all proper $a\in P$. A weight $w$ is used for defining a map $\diag{w}:\RR^B\to\RR^B$ via the pointwise product $\diag{w}(x)=wx$.
\end{definition}
Given a weight $w$ on $P$ we revisit the map $\rho:P^\circ\to\RR^B$, only that now we view it as a piecewise affine map of $\rho:S(P)\to\RR^B$. Extending the construction from \cite{Niblo_Roller-property_T} slightly, we define a new embedded complex in $\RR^B$.
\begin{definition}[$\ell_1$ realization of $\cube{}{P}$] Let $P$ be a discrete poc set with weight $w$ and basepoint $B\in P^\circ$. The {\it weighted dual $\cube{w}{P}_B$ of $P$ (with weight $w$ and basepoint $B$)} is defined to be the image of $\cube{}{P}_B$ under the map $\rho_w:\cube{}{P}\to\ell_1(B)$ obtained as the affine extension $\diag{w}\circ\rho:P^\circ\to\RR^B$, endowed with the metric $\ellone{w}$ induced on it from $\ell_1(B)$. 
\end{definition}
Some important observations:
\begin{proposition}[Properties of $\cube{w}{P}_B$]\label{properties of the weighted dual} Let $P$ be a discrete poc set with non-degenerate weight $w$ and basepoint $B\in P^\circ$. If $\cube{}{P}_B$ is locally finite, then:
\begin{enumerate}
	\item $\rho_w$ is a homeomorphism of $\cube{}{P}_B$ onto $\cube{w}{P}_B$. In particular, $\cube{w}{P}_B$ is a cubing.
	\item $\ellone{w}$ is a piecewise-$\ell_1$ metric: $\ellone{w}$ coincides with the quotient metric on $\cube{w}{P}_B$ obtained by endowing each cube with the metric induced on it from $\ell_1(B)$ (and then carrying out the appropriate identifications).
	\item $\rho_w$ restricts to a median-preserving map of $P^\circ$ into $\ell_1(B)$.
	\item $\cube{w}{P}_B$ is a median metric space -- in fact, the smallest geodesic median subspace of $\ell_1(B)$ containing $\rho_w(P^\circ)$.
\end{enumerate}
\end{proposition}
\begin{proof} To verify (1) we will use the local finiteness of $\cube{}{P}_B$. As $\cube{1}{P}_B$ and $\cube{w}{P}_B$ are related by the bijective stretching map $\diag{w}$, it would suffice to verify that the restriction of $\diag{w}$ to $\cube{1}{P}_B$ is a homeomorphism onto $\cube{w}{P}_B$. Recall that a mapping $f:X\to Y$ between topological spaces is continuous if and only if, given a locally finite cover of $X$ by {\it closed} sets, $f$ restricts to a continuous map on each element of the given cover. Thus, to verify the continuity of $\diag{w}$ and $\diag{w}\inv$ it suffices to verify their continuity when restricted to (each) single cube of $\cube{1}{P}_B$ and $\cube{w}{P}_B$, respectively. Since all the cubes inquestion are finite-dimensional, we are done.

To see (2), observe that a cube of $\cube{}{P}$ is mapped to a rectangular parallelopiped with edges parallel to the coordinate axes (RAP). Since paths that are piecewise parallel to the coordinate axes are geodesics in $\ell_1(B)$ provided they cross each hyperplane at most once, the quotient metric induced from endowing each RAP with the trace metric from $\ell_1(B)$ coincides with the distance measured along geodesics of $\ell_1(B)$ which do not exit $\cube{w}{P}_B$. The connectedness of $\cube{w}{P}_B=\rho_w(\cube{}{P}_B)$ finishes the proof.

Property (3) follows from \cite{Chatterji_Drutu_Haglund-measured_spaces_with_walls}, Lemma 3.12 and Theorem 5.1.

Property (4) is the tricky one. Our argument for (2) explains the fact of $\cube{w}{P}_B$ being a geodesic subspace of $\ell_1(B)$. One observes:
\begin{itemize}
	\item[-] Given points $x,y,z\in\cube{w}{P}_B$, Since $P$ is discrete we must have $x(a)=y(a)=z(a)$ for all but finitely many $a\in P$, reducing the problem to the case when $P$ is finite;
	\item[-] For any pair of vertices $U,W\in P^\circ$ sharing a cube in $\cube{}{P}_B$, if $Q$ is the minimal cube containing both $U$ and $W$ then $\rho_wQ=I(\rho_wU,\rho_wW)$, the interval calculated in $\ell_1(B)$, is the unique RAP with antipodal vertices $\rho_wU$ and $\rho_wW$. This explains the minimality property claimed in (4).
\end{itemize}
It remains to verify that the $\ell_1(B)$-median $\med(x,y,z)$ of a triple of points $x,y,z\in\cube{w}{P}_B$ is also contained in $\cube{w}{P}_B$. This verification is purely technical, and we omit the details in the interest of saving space, leaving only an outline of the argument: knowing that $\rho_w:P^\circ\to\ell_1(B)$ is median-preserving, one considers the set $A$ of all $a\in B$ for which at least one of $x(a),y(a),z(a)$ does not belong to $\{0,w(a)\}$. If $A$ is empty, then $x,y,z$ are vertices of $\cube{w}{P}_B$ -- images of points in $P^\circ$ under $\rho_w$, that is -- and there is nothing to prove. If $A$ is non-empty, observe that $A$ is a transverse set. Considering $m=\med(x,y,z)$ as a real-valued function of $B$ and recalling medians in $\ell_1(B)$ are computed coordinate-by-coordinate, we conclude that the value of $m$ on $B\minus A$ is determined by majority (as in the $A=\varnothing$ case), thus forcing $m$ into the unique cube whose vertices coincide with $m$ on $B\minus A$ and have arbitrary values $m(a)\in\{0,w(a)\}$ for all $a\in A$.
\end{proof}
\begin{remark} From (1) it is now clear that replacing the geometries of individual cubes with a different $\ell_p$ geometry does not change the homeomorphism type of the constructed space. Thus, $\cube{w}{P}_B$ can serve as a realization of $\cube{}{P}_B$ for any of the allowed piecewise geometries and for any choice of weight $w$ as long as $\cube{}{P}_B$ is locally finite.
\end{remark}
\begin{definition} By a piecewise-$\ell_1$ cubing we mean a space of the form $\cube{w}{P}_B$ endowed with the metric $\ellone{w}$. This metric space will be denoted $\cubeone{w}{P}$. In the special case when $w(a)=1$ for all proper $a\in P$, we will identify $\cubeone{w}{P}=\cubeone{}{P}$, and refer to $\ellone{w}$ simply as $\ellone{}$.
\end{definition}

\begin{figure}[t]
	\begin{center}
		\includegraphics[width=.6\columnwidth]{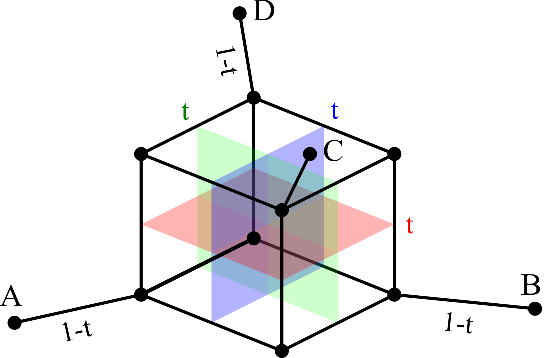}
	\end{center}
	\caption{The space $X_t$ from Example~\ref{example:non-injective 1}, constructed as a weighted piecewise-$\ell_1$ cubing dual to the poc set $P$ from Example~\ref{example:non-injective 2}. Note how the walls drawn inside the three-dimensional cube match the non-singleton halfspaces of the poc set $P$.\label{fig:non-injective median space 2}}
\end{figure}

\begin{example}\label{example:non-injective 2} Revisiting Example~\ref{example:non-injective 1}, we represent the space $X_t$ (figure \ref{fig:non-injective median space}) from that example as a piecewise-$\ell_1$ cubing. First, the poc-set structure may be chosen to have the form $P=S\cup S^\ast\}$ where
\begin{displaymath}
	S=\left\{\varnothing,
		\{A\},\{B\},\{C\},\{D\},
		\{A,B\},\{A,C\},\{A,D\}	
	\right\}
\end{displaymath}
with $P$ considered a sub poc set of $\power{\{A,B,C,D\}}$. Next we set the weights to equal
\begin{displaymath}
	w(\text{singleton})=1-t\,,\;\;
	w(\text{pair})=t.
\end{displaymath}
Note how the weights of walls separating any given pair of points in $\{A,B,C,D\}$ add up to $2$---see figure \ref{fig:non-injective median space 2}.
\end{example}

\subsection{Halfspaces and Walls of the Weighted Realization} 
Understanding the open (closed) half-spaces of $\cube{w}{P}_B$ for non-degenerate $w$ is easy: they are the intersections of $\cube{w}{P}_B$ with the open (closed) halfspaces of $\ell_1(B)$. The latter all have the form
\begin{gather}
	\half{b}(t):=\set{x\in\cube{w}{P}_B}{x(b)<t}\notag\\
	\text{or}\\
	\half{b^\ast}(t):=\set{x\in\cube{w}{P}_B}{x(b)>t}.\notag
\end{gather}
Note how, for any choice of $0<t<w(b)$ and $b\in B$, one has
\begin{equation}
	\rho_w(P^\circ)\cap\half{b}(t)=V(b)
\end{equation}
so that the `abstract' halfspaces---the elements of $P$---are reconstructed from the `visual' half-spaces of the realization. It is common in the field to refer to the sets of the boundaries of half-spaces
\begin{equation}
	\wall{b}(t)=\wall{b^\ast}(t)=\set{x\in\cube{w}{P}}{x(b)=t}\,,\qquad
	0<t<w(b)
\end{equation}
as the {\it walls}, or {\it hyperplanes}, of the cubical complex $\cube{w}{P}_B$, noting that every wall separates $V(b)$ from $V(b^\ast)$ in $\cube{w}{P}_B$, while having a `thickness' of $w(b)$. The distances between vertices are influenced accordingly: the following expression for the distance between vertices $U,W\in P^\circ$ is derived directly from property (2) stated in Proposition~\ref{properties of the weighted dual}:
\begin{equation}\label{ellone formula}
	\ellone{w}(\rho_wU,\rho_wW)=\sum_{a\in W\minus U}w(a)=\sum_{b\in U\minus W}w(b)\,,
\end{equation}
as $U\minus W$ indexes the set of walls separating $U$ from $W$.

\medskip
For a general pair of points $x,y\in\cube{w}{P}_B$ a similar formula may be written down. 
\begin{definition}[separators]\label{defn:separator} Let $x,y\in\cube{w}{P}_B$. The {\it separator} of $x$ and $y$ is the set---denoted $x\minus y$ by abuse of notation---of all $a\in P$ such that either
\begin{itemize}
	\item $a\in B$ and $x(a)<y(a)$;
	\item $a^\ast\in B$ and $x(a)>y(a)$.
\end{itemize}
\end{definition}
\begin{remark} Note how, if $w$ is non-degenerate, one always has $\rho_wU\minus\rho_wW=U\minus W$ for $U,W\in P^\circ$. Thus, the separator of a pair of vertices is the same as their combinatorial separator in $P$.
\end{remark}
The distance formula then trivially becomes:
\begin{equation}\label{general ellone formula}
	\ellone{w}(x,y)=\sum_{a\in x\minus y}\left|y(a)-x(a)\right|.
\end{equation}
There are two cases to consider for the summands:
\begin{itemize}
	\item[-] $x(a)=0$ and $y(a)=w(a)$ produces a summand of $w(a)$. In this case we see that {\it every} wall of the form $\wall{a}(t)$ separates $x$ from $y$ -- hence the contribution of $w(a)$ to the distance between the points.
	\item[-] $0<x(a)<y(a)<w(a)$. In this case, the two points are contained in a thickened hyperplane of $\cube{w}{P}_B$, which is merely a direct product of an interval with the dual cubing of the sub poc set $P_{\pitchfork a}$ of all $b\in P$ satisfying $b\pitchfork a$.
\end{itemize}

\subsection{Degenerate Weights}\label{degenerate weights} Some additional care is required for the case when some of the weights on a cubing are set to zero. This degenerate case is essentially identical to the non-degenerate one, as one may think of a null weight assigned to a wall as the limiting result of shrinking the `thickness' of that wall to zero. More formally, if the given weight $w$ is degenerate, the metric $\ellone{w}$ becomes a pseudo-metric. To obtain the reduction to the non-degenerate case, set:
\begin{equation}
	Z=\set{a\in P}{a\neq\minP,\maxP\text{ and }w(a)=0}\,,\quad
	\bar{P}:=P\minus Z\,,\quad
	\bar{w}=w\res{\bar{P}}.
\end{equation}
Inspecting the situation yields that $inc:\bar{P}\to P$ is a poc morphism, and that its dual $inc^\circ:P^\circ\to \bar{P}^\circ$, known as the {\it co-restriction map}, is given by $U\mapsto U\minus Z$. It is straightforward to prove that the (unique) piecewise affine extension of $inc^\circ$ induces an isometry of $\hau{\cube{w}{P}}$ onto $\cube{\bar{w}}{\bar{P}}$.

\subsection{Summability Concerns}\label{summability}
An interesting phenomenon takes place in the transition from unit piecewise-$\ell_1$ cubings (realized as the form $\cube{1}{P}_B$) to their weighted realizations $\cube{w}{P}_B$, having to do with the possible presence of nested sequences with summable weights.

Consider the following example (compare with~\cite{Bridson_Haefliger-nonpositive_curvature}, Example 7.11). 
\begin{example}[locally finite weighted cubing that's not complete]\label{example:incomplete} Let $P$ be the poc set generated by the natural numbers with their standard ordering, that is: $P$  consists of the symbols $\minP,\maxP$, and distinct elements $n,n^\ast$ for each $n\in\NN$, with an appropriately defined $\ast$-operator and subject to the mandated properties of poc sets. 
Also consider the weight $w(n)=w(n^\ast)=2^{-n}$, and pick the basepoint $B=\NN\cup\{\maxP\}$. 
Then $\cube{w}{P}_B$ is realized in $\ell_1(B)$ as a union of intervals $I_n$, $n\in\NN$ where each $x\in I_n$ has the form:
\begin{equation}
	x(k)=\left\{\begin{array}{cl}
		\tfrac{1}{2^k}	&\text{if }k<n\\[.25em]
		t\in[0,\tfrac{1}{2^n}]	&\text{if }k=n\\[.25em]
		0	&\text{if }k>n
	\end{array}\right.
\end{equation}
(We drop the coordinate labelled $\maxP$, as all $U\in P^\circ$ contain it). 
The endpoints of the various $I_n$ are precisely ultra-filters in $P^\circ$ which are almost-equal to $B$, with the only element of $P^\circ$ left unaccounted for being $C=\NN^\ast\cup\{\maxP\}$, which, due to the summability of the weight sequence, could be realized in $\ell_1(B)$ as $c:\NN\to\RR$, $c(k)=\tfrac{1}{2^k}$.
We conclude that $\cube{w}{P}_B$ is isometric to the interval $[0,1)$, while adding the point of $\ell_1(B)$ corresponding to the ``vertex at infinity'', $C$, yields the completion of $\cube{w}{P}_B$.

This situation should be compared with the unit cubing realization, where $\rho_w(C)$ ends up being the point $c(k)=1$, $k\in\NN$, and the intervals $I_n$ are edges of the unit infinite cube. While the union of the $I_n$ does lie in $\ell_1(B)$, the point $c=\rho_w(C)$ does not, and we observe that $\cube{w}{P}_B$ is complete, being isometric to the interval $[0,\infty)$.
\end{example}

The last example demonstrates, among other things, how the introduction of variable weights opens up a way for the arguments at the end of Section~\ref{section:geometric complexes} to fail (now, that there are infinitely many isometry types of cells in the cubing), resulting in $\cube{w}{P}_B$ being incomplete.

This raises a question, made relevant by Roller's work on the topological structure of $P^\circ$ (where $P$ is a discrete poc set) as a Stone median algebra with respect to the Tychonoff topology (see \cite{Roller-duality}, Section 5). In this context, $P^\circ$ is thought of as a subspace of the Tychonoff product $\power{P}$, in turns out to be closed. As a result, if $X\subset P^\circ$ is an almost-equality class, then the Tychonoff closure $\bar{X}$ of $X$ in $P^\circ$ is compact.

In the piecewise-$\ell_1$ case, the metric completion of $Z:=\cube{w}{P}_B$ is obtained by taking the closure of $\bar{Z}$ in $\ell_1(B)$. At the same time, realizing that convergence in the Tychonoff topology is coordinatewise convergence, we conclude that every point of $\bar{Z}$ that is a limit of vertices $(U_n)\subset\cube{}{P}_B$ must itself correspond to an element of $P^\circ$ arising as the Tychonoff limit of the $U_n$. Skipping ahead a little bit, we note that convergence in  weighted piecewise-$\ell_p$ realizations of $\cube{}{P}_B$ implies Tychonoff convergence {\it for any choice of $p\in[1,\infty]$}, but it is not immediately clear that the metric completion may be computed in the same way. Moreover, it is easily shown (\cite{Guralnik-Roller_boundary}, Corollary 3.10) that the Tychonoff closure of an almost-equality class is a union of almost-equality classes. All together, these observations motivate the following question, for which we expect a positive answer, at least in the case when $p=1$ and with possible finiteness restrictions on $\cube{}{P}$ (e.g. $\cube{}{P}$ locally finite, or, more generally, $P$ contains no infinite transverse set):
\begin{question} Let $p\in[1,\infty]$, and suppose $P$ is a discrete poc set with weight $w$, and let $B\in P^\circ$ be a basepoint. To what extent is it true that the metric completion of the weighted piecewise-$\ell_p$ cubing $\cube{w}{P}_B$ may be obtained as the intersection of $\ell_p(B)\subset\RR^B$ with the realization $\rho_w(\cube{}{P})\subset\RR^B$? (Keep in mind that the latter is the union of {\it all} the connected components of $\cube{}{P}$, realized in $\RR^B$ via the map $\rho_w$ determined by the basepoint $B$.)
\end{question}

\section{Piecewise-$\ell_\infty$ Cubings} We have finally arrived at the point where a formal and workable construction of piecewise-$\ell_\infty$ cubings is given.

\subsection{Generalities} In view of Sageev--Roller duality (Theorem~\ref{Sageev-Roller duality}) the definition of a piecewise-$\ell_\infty$ cubing from Section~\ref{section:geometric complexes} may be rewritten as follows.
\begin{definition} By a piecewise-$\ell_\infty$ cubing we mean a metric space of the form $\cube{w}{P}_B$, where $P$ is a discrete poc-set, $w$ is a non-degenerate weight and $B\in P^\circ$ is a basepoint. This time we endow $\cube{w}{P}_B$ with the quotient metric $\ellinfty{w}$ obtained by endowing each of the cubes of $\cube{w}{P}$ with the metric induced on it from $(\RR^B,\norm{\cdot}_\infty)$. The resulting quotient space will be denoted $\cubeinfty{w}{P}_B$. By a {\it unit} piecewise-$\ell_\infty$ cubing we mean a space of the form $\cubeinfty{w}{P}_B$ with $w(a)=1$ for every proper $a\in P$. We denote this space with $\cubeinfty{}{P}_B$, and its metric with $\ellinfty{}$.
\end{definition}
The discussion at the end of Section~\ref{section:geometric complexes} provide the following result for the case of finite and locally-finite cubings, respectively:
\begin{lemma} Let $P$ be a finite poc set with weight $w$ and basepoint $B\in P^\circ$. Then the metric $\ellinfty{w}$ on $\cube{w}{P}_B$ is the length metric induced on $\cube{w}{P}_B$ from $(\RR^B,\norm{\cdot}_\infty)$. Furthermore, $\cubeinfty{w}{P}_B$ is a complete geodesic metric space.\ep
\end{lemma}
\begin{corollary}\label{cor:ellinfty is a length metric} Let $P$ be a discrete poc set with weight $w$ and basepoint $B\in P^\circ$. If the dual $\cube{}{P}_B$ is locally finite, then $\ellinfty{w}$ is the length metric on $\cube{w}{P}_B$ induced on it from $(\RR^B,\norm{\cdot}_\infty)$, and $\cubeinfty{w}{P}_B$ is a geodesic metric space.\ep
\end{corollary}
We emphasize that in the finite case one has $\cube{}{P}_B=\cube{}{P}$, and that in either case the resulting metric is independent of the choice of basepoint $B$ in the cubing.

\subsection{Lower Bound on $\ellinfty{w}$} The following technical lemma places a lower bound on distances in $\cubeinfty{w}{P}$.
\begin{lemma}\label{separation lemma} Let $U,W\in P^\circ$ be vertices of $\cube{w}{P}$ and let $N$ be a finite nested subset of $U\minus W$. Then 
\begin{equation}\label{eqn:separation bound}
	\ellinfty{w}(\rho_wU,\rho_wW)\geq\sum_{a\in N}w(a):=w(N)\,.
\end{equation}
\end{lemma}
\begin{proof} For any $a\in P$, the sub-complex $X(a)$ of $X=\cube{w}{P}$ induced by the vertex set $V(a)\subset P^\circ$ is itself a cubing. We start by verifying that $\ellinfty{}(x,y)\geq w(a)$ for {\it any} $x\in X(a)$ and $y\in X(a^\ast)$: since the walls $\wall{a}(t)$, $t\in(0,w(a))$ separate $X(a)$ from $X(a^\ast)$ in $X$, every string $\mbf{q}=(x_0,x_1,\ldots,x_n)$ from $x$ to $y$ in $X$ must satisfy 
\begin{equation}
	\length{\mbf{q}}\geq \sum_{i=1}^n |x_i(a)-x_{i-1}(a)|\geq |x_0(a)-x_n(a)|\geq w(a)\,,
\end{equation}
by the triangle inequality.

\medskip
Now, if $N\subset U\minus W$ is nested, then for any $a,b\in N$ we {\it cannot} have $a\leq b^\ast$, since $a,b\in U$ and $U$ is coherent. Neither can we have $a^\ast\leq b$, for $a^\ast,b^\ast\in W$ and $W$ is coherent. We are left with $a\leq b$ or $b\leq a$ for all $a,b\in N$. Thus, when $N$ is finite we may write $N=\{a_1,\ldots,a_n\}$ with $a_1<a_2<\ldots<a_n$, which gives: 
\begin{equation}
	\begin{array}{rcccl}
	\rho_wU&\in& X(a_1)\subset X(a_2)\subset\cdots\subset X(a_n)&&\\[.5em]
	&&X(a_1^\ast)\supset X(a_2^\ast)\supset\cdots\supset X(a_n^\ast) &\ni&\rho_wW.
	\end{array}
\end{equation}
The fact that each $X(a_i)$ and $X(a_j^\ast)$ is an $\ell_\infty$-cubing in its own right allows us to repeatedly apply the preceding argument to conclude 
\begin{equation}
	\ellinfty{w}(\rho_wU,\rho_wW)\geq\sum_{i=1}^nw(a_i)
\end{equation}
as desired.
\end{proof}
An immediate corollary is the following.
\begin{corollary} Let $P$ be a discrete poc set with weight $w$. Then the inclusion map 
\begin{displaymath}
	inc_{_Q}:(Q,\norm{\cdot}_\infty)\to(\cube{w}{P},\ellinfty{w})
\end{displaymath}
is an isometric embedding for every face $Q$ of $\cube{w}{P}$.\ep
\end{corollary}
Recall that a geodesic in a metric space $(X,d)$ from a point $x$ to a point $y$ is an isometric embedding $\gamma:[0,d(x,y)]\to X$ satisfying $\gamma(0)=x$ and $\gamma(d(x,y))=y$. A discretized version of this notion for our purposes is the following definition.
\begin{definition} We say that a string $\mbf{p}$ from $x$ to $y$ in $\cube{w}{P}$ is {\it geodesic}, if $\length{\mbf{p}}=\ellinfty{w}(x,y)$. 
\end{definition}
The proof of the last lemma produces an obvious lower bound on the length of a geodesic between two vertices of $\cube{w}{P}$.
\begin{corollary}\label{lower bound on distance} Let $P$ be a discrete poc set with weight $w$. Then the bound
\begin{equation}
	\ellinfty{w}(\rho_wU,\rho_wW)\geq
	\max\set{w(N)}{N\text{ is a chain in }U\minus W}
\end{equation}
holds for any $U,W\in P^\circ$.\ep
\end{corollary}

\subsection{Constructing Geodesics in Unit $\ell_\infty$-cubings} Assume for now that all the weights on $P$ are unity. We now consider a special family of strings introduced in \cite{Niblo_Reeves-biautomatic}.
\begin{definition} A string $(U_0,\ldots,U_m)$ of vertices in $\cube{w}{P}$ is said to be a {\it normal cube path}, if
\begin{equation}
	U_i\minus U_{i+1}=\min(U_i\minus U_m)
\end{equation}
for all $i=0,\ldots,m-1$.
\end{definition}
It is easy to verify that $\min(U\minus V)$ is a transverse subset of $P$ for any pair of vertices $U,V\in P^\circ$. Thus, a normal cube path from $U$ to $V$ has the form:
\begin{equation}
	U_0=U\,,\;
	U_1=\flip{\min(U_0\minus V)}{U_0}\,,\;\ldots\;,\,
	U_{k+1}=\flip{\min(U_k\minus V)}{U_k}\,,\;\ldots
\end{equation}
and in particular:
\begin{itemize}
	\item Normal cube paths are taut;
	\item Each consecutive pair of vertices along a normal cube path are at unit distance from each other (as they span the diagonal of an embedded unit $\ell_\infty$-cube);
	\item Thus, the length $\length{\mbf{p}}$ of a normal cube path $\mbf{p}$ as in the definition above equals $m$.
\end{itemize}
These two properties hint at the possibility that a normal cube path is, in fact an $\ell_\infty$-geodesic. In any case, the length of such a path provides one with an upper bound on the distance between its endpoints. We use normal cube paths to prove the following:
\begin{proposition}\label{length of normal cube path} For every $U,W\in P^\circ$ one has the formula
\begin{equation}\label{eqn:unit ellinfty formula}
	\ellinfty{}(\rho U,\rho W)=\max\set{|N|}{N\subseteq U\minus W\text{ is nested}}.
\end{equation}
\end{proposition}
\begin{proof} Let $\mbf{p}=(U_0,\ldots,U_m)$ be the normal cube path from $U_0=U$ to $U_m=W$. In $\cubeinfty{}{P}$ this implies $\length{\mbf{p}}=m$. Setting $A_{i+1}=U_i\minus U_{i+1}$ for $i=0,\ldots,m-1$ we have:
\begin{enumerate}
	\item $A_i$ is transverse, for all $i=1,\ldots,m$;
	\item $U\minus W=\bigcup_{i=1}^m A_i$ and this union is disjoint.
\end{enumerate}
In particular, any nested set $N\subseteq U\minus W$ intersects every $A_i$ in at most one element, and therefore satisfies $|N|\leq m$.

\medskip
Now for any $i>1$ and any $a_i\in A_i$ we observe that $a_i\in U_{i-1}\minus U_m$. At the same time, by the definition of a normal cube path,
\begin{equation}
	A_{i-1}=U_{i-1}\minus U_i=\min(U_{i-1}\minus U_m)
\end{equation}
and we conclude that there has to be some $a_{i-1}\in A_{i-1}$ with $a_{i-1}<a_i$. Starting with any $a_m\in A_m$ we thus obtain at least one chain $N=\{a_1,\ldots,a_m\}$ with $a_i\in A_i$ for all $i$. Thus, $U\minus W$ contains a nested subset of cardinality $m$. We conclude:
\begin{equation}
	\ellinfty{}(U,W)\leq\length{\mbf{p}}=\max\set{|N|}{N\text{ is a nested subset of }U\minus W}.
\end{equation}
The reverse inequality was established previously, in Corollary~\ref{lower bound on distance}, so we are done.
\end{proof}
As a by-product of this proof we also obtain the following corollary.
\begin{corollary}\label{normal implies geodesic} In a unit piecewise-$\ell_\infty$ cubing, normal cube paths are geodesic strings.
\end{corollary}

\subsection{Subdivisions of $\ell_p$ Cubings}\label{section:subdivision} We are looking for an operation on poc sets that would result in subdividing each cube of the dual cubing into a `grid' of smaller cubes.

\medskip
\begin{definition}[Refinement] We will say that a morphism of poc sets $f:\tilde P\to P$ is a {\it refinement map} -- and that $\tilde P$ is a refinement of $P$---if $f\inv(\minP)=\{\minP\}$, and for any proper pair $p,q\in P$ one has:
\begin{enumerate}
	\item $f\inv(p)$ is a finite chain in $\tilde P$,
	\item for any $\tilde p\in f\inv p$ and $\tilde q\in f\inv q$ one has $\tilde p<\tilde q\IFF p<q$.
\end{enumerate}
\end{definition}

\medskip
We will now show that this notion of refinement produces the correct notion of subdivision in the dual. The idea is, roughly, that---since each cube $Q$ of $\cube{}{P}$ is characterized by a transverse set of walls (Lemma~\ref{lemma:cubes and transversality})---if each of these walls were to be replicated into a linear sub poc set (in $\tilde P$), then the dual of the union of these linear refinements is the cartesian product of subdivided intervals; i.e., a subdivided cube (see Examples~\ref{Cartesian products} and \ref{dual of linear}, respectively).
\begin{proposition}[Subdivision Lemma]\label{subdivision} Let $f:\tilde P\to P$ be a refinement of a discrete poc set $P$ and let $w,\tilde w$ be weights on $P$ and on $\tilde P$ respectively, satisfying $w(a)=\sum_{c\in f\inv(b)}\tilde w(c)$ for all $a\in P$. Then for any choice of $B\in P^\circ$, setting $\tilde{B}=f\inv(B)\in \tilde P^\circ$ (see R	emark~\ref{remark:dual maps}) there is a bijection $F$ of $\cube{w}{P}_B$ onto $\cube{\tilde w}{\tilde P}_{\tilde B}$ with the following properties:
\begin{enumerate}
	\item $F$ extends the dual map $f^\circ:P^\circ\to\tilde P^\circ$ given by $f^\circ(U)=f\inv(U)$;
	\item $F$ is an isometry of $\cubeone{w}{P}$ onto $\cubeone{\tilde w}{\tilde P}$ (in particular, $F$ is a median isomorphism);
	\item For each proper $a\in P$, if one writes $f\inv(a)=\{c_1,\ldots,c_n\}$, $n=n(a)$, with $c_1<\ldots< c_n$ then there are reals $0<t_1<\ldots<t_n<w(a)$ such that $F(\half{a}(t_k))=\half{c_k}(\tfrac{\tilde w(c_k)}{2})$ for all $k\in\{1,\ldots,n\}$;
	\item $F$ is an isometry of $\cubep{w}{P}_B$ onto $\cubep{\tilde w}{\tilde P}_B$ for any $p\in[1,\infty]$.
\end{enumerate}
\end{proposition}
\begin{proof} Let $X=\cube{w}{P}_B$ and $\tilde X=\cube{\tilde w}{\tilde P}_{\tilde B}$. For each $a\in P$, writing $f\inv(a)=\{c_1,\ldots,c_n\}$ as above, set $t_1=\tfrac{\tilde w(c_1)}{2}$ and $t_{k+1}=\sum_{i=1}^k\tilde w(c_i)+\tfrac{\tilde w(c_{k+1})}{2}$ for $1\leq k<n$. Let $\hsm{H}$ denote the set of half-spaces of $(X,\ellone{w})$ arising as $\half{}(a,k):=X\cap\half{a}(t_k)$ for all $a\in P$ and $c_k\in f\inv(a)$ (as above), augmented with $\{\varnothing,X\}$, to make it into a sub poc set of $\power{X}$ with respect to inclusion and the complementation operator $a\mapsto X\minus\cl{a}$. Observe that the map sending each $\half{}(a,k)$ to $a$ is a refinement map factoring as the composition of $f:\tilde P\to P$ over the isomorphism sending each $\half{}(a,k)\in\hsm{H}$ to the appropriate $c_k\in\tilde P$. In particular, the poc sets $\tilde P$ and $\hsm{H}$ have the same dual median space. Since $\hsm{H}$ is a half-space system in the median space $(X,\ellone{w})$, its dual median space is isometric to $(X,\ellone{w})$ (see~\cite{Chatterji_Drutu_Haglund-measured_spaces_with_walls}, Theorem 5.12 and Lemma 3.12). This takes care of the first three requirements.

In order to verify the fourth property of $F$, we observe that, indeed, $F$ induces by pullback a cartesian subdivision of each cube of $\cube{w}{P}_B$. Any such subdivided cube is isometric, as a piecewise-$\ell_p$ cubical complex to the original cube without the subdivision. Since no two faces of $\cube{\tilde w}{\tilde P}_{\tilde B}$ get identified whose preimages under $F$ were not identified in $\cube{w}{P}_B$, we conclude that $F$ is an isometry of piecewise-$\ell_p$ complexes.  
\end{proof}

\section{Deforming Piecewise-$\ell_p$ Cubings}
We begin with a rather coarse estimation of the Gromov--Hausdorff distance between $\ell_p$ cubings sharing the same combinatorial structure:
\begin{lemma}\label{ellinfty deformation GH distance bound} Fix $p\in[1,\infty]$. Let $P$ be a finite poc set and $B\in P^\circ$ be a basepoint. Let $u,w$ be non-degenerate weights on $P$. Then the pair of spaces $X=\cubep{u}{P}_B$ and $Y=\cubep{w}{P}_B$ admits an $\epsilon$-approximation for any $\epsilon>\norm{u-w}_1$.
\end{lemma}
\begin{proof} Denote $X=\cubep{u}{P}$, $Y=\cubep{w}{P}$, $\epsilon=\norm{u-w}_1$ and let $\delta>0$. 
Without loss of generality we may assume that the $\ell_p$-diameter of a cube in either space does not exceed $\delta$: otherwise, construct a refinement $f:\tilde P\to P$ admitting weights $\tilde u$ and $\tilde w$ with the property that (1) $|\tilde u(\tilde a)-\tilde w(\tilde a)|\leq|u(a)-w(a)|$ for every $a\in\tilde P$, and (2) every cube of $\tilde X=\cubep{\tilde u}{\tilde P}_{\tilde B}$ and of $\tilde Y=\cubep{\tilde w}{\tilde P}$ has $\ell_p$-diameter no more than $\delta$.
Proposition~\ref{subdivision} tells us we may replace $X$ and $Y$ by the spaces $\tilde X$ and $\tilde Y$ realizing the refinement, respectively, while replacing $P$ by $\tilde P$ and $\epsilon$ by the smaller $\epsilon':=\norm{\tilde u-\tilde w}_1\leq\epsilon$.

We proceed to define a relation $R\subseteq X\times Y$ by setting $(x,y)\in R$ if and only if the points $\rho\inv_u(x)$ and $\rho\inv_w(y)$ share a cube in $\cube{}{P}_B$.

Let $\mbf{p}=(x_0,\ldots,x_n)$ be any taut string in $X$. Then $\mbf{q}=(y_i:=\rho_w\rho\inv_u(x_i))$ is a taut string in $Y$, and both may be written as $\mbf{p}=\rho_u(\mbf{t})$, $\mbf{q}=\rho_w(\mbf{t})$ for a taut string in $\cube{}{P}_B$. One may write:
\begin{eqnarray*}
	\length{\mbf{p}}
		&=&\sum_{i=1}^n\left\Vert x_i-x_{i-1} \right\Vert_p
		=\sum_{i=1}^n\left\Vert u\cdot|t_i-t_{i-1}| \right\Vert_p
\end{eqnarray*}
and similarly for $\mbf{q}$. Then:
\begin{eqnarray*}
	\left|\length{\mbf{p}}-\length{\mbf{q}}\right|
		&\leq&\sum_{i=1}^n\left|
			\left\Vert u\cdot|t_i-t_{i-1}|\right\Vert_p
			-\left\Vert w\cdot|t_i-t_{i-1}|\right\Vert_p
		\right|\\
		&\leq&\sum_{i=1}^n\left\Vert |u-w|\cdot|t_i-t_{i-1}| \right\Vert_p\\
		&\leq&\sum_{i=1}^n\left\Vert |u-w|\cdot|t_i-t_{i-1}| \right\Vert_1\\
		&=&\sum_{i=1}^n\sum_{a\in B}|u(a)-w(a)|\cdot|t_i(a)-t_{i-1}(a)|\\
		&=&\sum_{a\in B}|u(a)-w(a)|\sum_{i=1}^n|t_i(a)-t_{i-1}(a)|\,,
\end{eqnarray*}
Now we make the observation that, if either of $\mbf{p}$, $\mbf{q}$, $\mbf{t}$ is a geodesic string, then each $t_i(a)$ is (weakly) monotone in $i$. For any of these cases we therefore obtain a bound of $1$ on the internal sum of the last expression, producing:
\begin{equation}\label{deformation bound}
	\left|\length{\mbf{p}}-\length{\mbf{q}}\right|
		\leq\Vert u-w\Vert_1\,.
\end{equation}

Let $f=\rho_w\circ\rho\inv_u:X\to Y$. Take points $x,x'\in X$ and $y,y'\in Y$ satisfying $(x,y)\in R$ and $(x',y')\in R$. If $\mbf{p}$ is a geodesic string in $X$ from $x$ to $x'$ of length $\lambda$, then $\mbf{q}=f(\mbf{p})$ is a string of length at most $\lambda+\epsilon$ in $Y$, by \eqref{deformation bound} above. This yields:
\begin{eqnarray*}
	\ellp{w}(y,y')
		&\leq& \ellp{w}(y,f(x))+\ellp{w}(f(x),f(x'))+\ellp{w}(f(x'),y')\\
		&\leq& \ellp{u}(x,x')+\epsilon+2\delta
\end{eqnarray*}
The result now follows by the symmetry of the argument.
\end{proof}
Thus, small deformations of the weight on a piecewise-$\ell_p$ cubing result in only small changes in the distances between its points. This allows us to apply our slightly deeper understanding of unit cubings to the study of the more general locally-compact case.
\begin{definition} Let $P$ be a discrete poc set endowed with a weight $w$. The {\it $n$-th lower rational approximation $\zapprox{n}{P,w}$ of $P$} is obtained as follows: For all $p\in P$ set $u(p)=\floor{n\cdot w(p)}$ and let $\zapprox{n}{P,w}$ denote the unique refinement of the weighted poc set $(P,u)$ with unit weights, rescaled by a factor of $\tfrac{1}{n}$.
\end{definition}
An immediate corollary of the preceding lemma is the following approximation lemma.
\begin{proposition}[Approximation Lemma]\label{prop:approximation lemma} Let $p\in[1,\infty]$ and let $P$ be a finite poc set with weight $w$. Fix a base point $B\in P^\circ$ and set $X=\cube{w}{P}_B$ with the associated picewise-$\ell_p$ metric. For each $n\in\NN$, let $X_n$ denote the dual of its $n$-th lower rational approximation, $n\in\NN$, endowed with the piecewise-$\ell_p$ metric and with the basepoint $B_n$ provided by the refinement map $f_n:\zapprox{n}{P,w}\to P$ as $B_n=f_n^\circ(B)=f_n^{-1}(B)$, by the subdivision lemma (Proposition~\ref{subdivision}). Then, the sequence $(X_n,B_n)$ converges to $(X,B)$ in the pointed Gromov--Hausdorff topology. In particular, in the $\ell_\infty$ case, if $x,y\in\cube{w}{P}_B$ and $x_n,y_n$ are vertices of $X_n$ with $x_n\to x$, $y_n\to y$ and $\mbf{p}_n$ is the normal cube path from $x_n$ to $y_n$, then the sequence $(\mbf{p}_n)_{n=1}^\infty$ converges to a geodesic arc in $X$ joining $x$ with $y$.
\end{proposition}
\begin{proof} For each $n$, let $u_n$ be the weight on $P$ defined by $u_n=\tfrac{1}{n}\left\langle n\cdot w\right\rangle$, and let $X_n'$ denote $\cube{u_n}{P}_B$ endowed with the piecewise-$\ell_p$ metric. Then, by the preceding lemma, the Gromov--Hausdorff distance between $X$ and $X_n'$ (both cubings dual to $P$, with basepoint $B$) does not exceed $\tfrac{|P|}{n}$, while $(X_n',B_n)$ and $(X_n,B_n)$ are related by a pointed isometry, by the subdivision lemma. We may therefore conclude that the sequence of the $(X_n,B_n)$ converges to $(X,B)$.
\end{proof}
As a corollary we obtain the promised reduction of our main result to the finite unit case.
\begin{corollary} If every finite unit piecewise-$\ell_\infty$ cubing is injective, then every finite piecewise-$\ell_\infty$ cubing is injective.
\end{corollary}
\begin{proof} Simply apply Lemma~\ref{proof of limit lemma} to the result of the last proposition.
\end{proof}
Applying Gromov--Hausdorff limits to exhaustions by finite convex sub-cubings one obtains:
\begin{corollary} If every finite unit piecewise-$\ell_\infty$ cubing is injective, then the completion of any locally finite piecewise-$\ell_\infty$ cubing is injective.
\end{corollary}
In particular, if the basepoint $B$ contains no maximal nested sequence with $\ell_1$-summable weights (e.g., in the simplest case, the weights are bounded away from zero), the resulting piecewise-$\ell_\infty$ cubing will be complete, and hence injective. However, if such a sequence exists, we are thrown back to the problem stated above, in the end of Section~\ref{summability}: is it the case that the metric completion of our cubing is not that hard to compute; that one needs only append a bunch of {\it additional} cubings to the $\cube{w}{P}_B$ in order to obtain an injective space?

\medskip
Finally, having reduced our problem to the case of finite unit cubings, let us see what we could glean from studying this special case using rational approximations. Clearly, $\zapprox{n}{P}$ for a unit-weighted finite poc set $P$ produces a refinement $P_n$ where each $d$-dimensional cube of $\cubeinfty{}{P}$ is subdivided into $n^d$ cubes of edge-length $\tfrac{1}{n}$. As a result, both $\cubeinfty{}{P}$ and $\cubeone{}{P}$ may be represented as Gromov--Hausdorff limits of the {\it discrete} rescaled spaces $K_n:=\tfrac{1}{n}(P_n^\circ,\ellinfty{})$ and $M_n:=\tfrac{1}{n}(P_n^\circ,\ellone{})$, respectively. As a consequence we obtain, for this special setting, the following corollary. 
\begin{corollary} Every closed ball $B(p,r)$ in $\cubeinfty{}{P}$ is the limit of a sequence of balls $B(p_n,r_n)\subset K_n$ with $r_n\geq r$. A closed subset $h\subset\cubeone{}{P}$ is a half-space if and only if it is the limit of a sequence of halfspaces $h_n\subset M_n$.
\end{corollary}
As a result we obtain an even deeper reduction of the main result to a finite, discrete problem.
\begin{corollary}\label{final reduction} If every finite poc set has the property
\begin{equation}
	\text{Any ball of integer radius in } (P^\circ,\ellinfty{}) \text{ is a convex subset of } (P^\circ,\ellone{}) \tag{$\dagger$}
\end{equation}
then every locally finite piecewise-$\ell_\infty$ cubing is an injective metric space.
\end{corollary}
\begin{proof} By the preceding corollary, if every finite poc set has $(\dagger)$, then every ball in any finite piecewise-$\ell_\infty$ cubing $X$ is $\ell_1$-convex. By Helly's theorem (Theorem~\ref{Helly theorem}), any finite family of pairwise-intersecting balls $X$ has a common point. Since $X$ is a geodesic space, hyper-convexity follows.
\end{proof}

\section{Proof of the Main Result, Final Remarks}
In view of Corollary~\ref{final reduction}, proving our main theorem requires merely the verification of property ($\dagger$) for finite poc sets $P$. 
\begin{lemma}\label{ball-separation lemma} Suppose $P$ is a finite poc set and let $U, W\in P^\circ$ be vertices with $\ellinfty{}(U,W)>n$ for some $n\in\NN$. Then there exists a wall of $P$ separating $W$ from the $\ell_\infty$-ball of radius $n$ about $U$. 
\end{lemma}
\begin{proof} Let $B$ denote the $\ell_\infty$-ball of radius $n$ about $U$. Consider the normal cube path $\mbf{p}=(U_0,\ldots,U_m)$, $m=\ellinfty{}(U,W)$ from $U$ to $W$. As in the proof of Proposition \ref{length of normal cube path}, let us write $A_i=U_i\minus U_{i-1}$ while observing that $U\minus W$ equals the {\it disjoint} union $A_1\cup\cdots\cup A_m$. Pick any $a\in A_m$ and construct a chain $a_1<\ldots<a_{m-1}<a_m=a$ with $a_i\in A_i$ -- again, as in the proof of Proposition~\ref{length of normal cube path}.

We contend that the wall $\{a,a^\ast\}$ is a wall of the kind we are looking for, that is:
\begin{enumerate}
	\item The vertex $W$ is contained in $V(a^\ast)$;
	\item The ball $B$ is contained in $V(a)$.
\end{enumerate}
Indeed, (1) is satisfied by construction: $a\notin W$ means $a^\ast\in W$, which, in turn, means $W\in V(a^\ast)$. 

Property (2) holds by Lemma~\ref{separation lemma}: indeed, any vertex in $V(a^\ast)$ is separated from $U$ by the nested set $\{a_1,\ldots,a_n\}$ and is therefore at $\ell_\infty$-distance at least $n$ from $U$. Since $V(a)=P^\circ\minus V(a^\ast)$ we conclude $B\subseteq V(a)$ and we are done.
\end{proof}
\begin{remark} Note that $P$ need not be finite for the assertion of the lemma to hold in the case when $U$ and $W$ lie in the same connected component of $\cube{}{P}$. By Corollary~\ref{interval}, $U$ and $W$ are contained in a finite sub-cubing, which is spanned by the median interval between them.
\end{remark}

\bigskip
To close, we would like to draw the reader's attention to an alternative line of reasoning having the added benefit of clarifying some leftover questions regarding the metric structure of a piecewise-$\ell_\infty$ cubing. Reading through our own exposition we felt that, in the end, one could not help but wonder at the effort we have put into avoiding a direct computation of $\ellinfty{w}$ in the general case. Why only the unit case?

\begin{figure}[t]
	\begin{center}
		\includegraphics[width=.5\columnwidth]{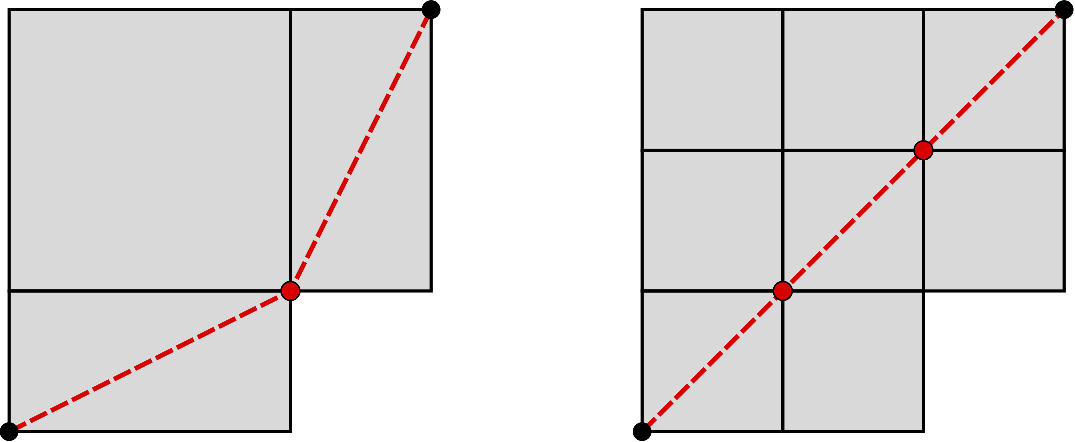}
	\end{center}
	\caption{An example of a non-geodesic normal cube path (left) in a piecewise-$\ell_\infty$ cubing versus a geodesic one in a refined (and hence isometric) cubing (right) -- see Example~\ref{example:bad normal cube path}.\label{fig:bad normal cube path}}
\end{figure}

\begin{example}\label{example:bad normal cube path} Figure \ref{fig:bad normal cube path} compares a weighted cubing drawn in the $\ell_\infty$ plane where the normal cube path joining a pair of points fails to be a geodesic string to a subdivided version of the same cubing, where the normal cube path between the same two points changes into a geodesic one, due to the weights being more uniformly distributed among the walls of the refined poc set.
\end{example}
This example suggests it should be possible to apply the Approximation Lemma (Proposition~\ref{prop:approximation lemma}) in a proof of an explicit formula for $\ellinfty{w}$ based on the formula for unit cubings. One can verify that this is indeed the case:
\begin{proposition} Let $P$ be a discrete poc set with non-degenerate weight $w$ and basepoint $B\in P^\circ$. Then, for any $x,y\in\cube{w}{P}_B$ one has
\begin{equation}
	\ellinfty{w}(x,y)=\max\set{\sum_{a\in N}|x(a)-y(a)|}{N\text{ is a nested subset of }x\minus y}
\end{equation}
(see Definition~\ref{defn:separator}).\ep
\end{proposition}

Skipping the proof, we would like to observe the fact that the convexity of balls in $\cubeinfty{w}{P}$ now becomes self-evident: if $y\notin B(x,r)$ in $\cubeinfty{w}{P}$, then the above formula provides us with a hyperplane of $\cubeone{w}{P}$ separating $y$ from $B(x,r)$ using essentially the same procedure as we had used in the proof of Lemma~\ref{ball-separation lemma}, avoiding the need for a reduction of the statement regarding the injectivity of piecewise-$\ell_\infty$ cubings to the finite, unit, vertex-only case.

This argument seems to apply Gromov--Hausdorff convergence more sparingly, but ultimately it does nothing but shift the weight (of the technical details) around. On an emotional note, we admit our preference of strategy was motivated by a sense of indebtedness to Isbell's vision: not only did he reveal the way (first followed by Mai and Tang), but he also provided the machinery (the duality theory of median algebras) for the present extension. It would have been ungrateful of us to have picked a different path.

\section{Acknowledgements} The authors gratefully acknowledge the support of Air Force Office of Science Research under the LRIR 12RY02COR, MURI FA9550-10-1-0567 and FA9550-11-10223 grants, respectively. We greatly appreciate the deep and thoughtful review this paper has received, and see ourselves indebted to the referee for multiple insightful comments.

\bibliographystyle{plain}

\end{document}